\RequirePackage[l2tabu,orthodox]{nag}		
\documentclass[reqno]{amsart}		
\usepackage[margin=1.35in,bottom=1.25in]{geometry}		
\usepackage[title,titletoc]{appendix}
\usepackage{lipsum}
\usepackage[foot]{amsaddr}		

\usepackage{algorithm}
\usepackage{algorithmic}

\usepackage{amsmath}		
\usepackage{amssymb}		
\usepackage{amsthm}		
\usepackage{mathtools}		
\usepackage{amsfonts}       

\usepackage{comment}
\usepackage{graphicx}
\usepackage{color}
\usepackage{wasysym}
\usepackage{bm} 

\usepackage[utf8]{inputenc} 
\usepackage[T1]{fontenc}    
\usepackage{hyperref}       
\usepackage{url}            
\usepackage{booktabs}       
\usepackage{nicefrac}       
\usepackage[kerning=true]{microtype}		
\usepackage{tablefootnote}

\usepackage{etoc}

\makeatletter
\newcommand{\newreptheorem}[2]{\newtheorem*{rep@#1}{\rep@title}
	\newenvironment{rep#1}[1]{\def\rep@title{#2 \ref*{##1}}\begin{rep@#1}}{\end{rep@#1}}
}
\makeatother

\newreptheorem{lemma}{Lemma}
\newreptheorem{theorem}{Theorem}
\newreptheorem{claim}{Claim}
\newreptheorem{proposition}{Proposition}
\newreptheorem{corollary}{Corollary}

\newcommand{\secref}[1]{Section~\ref{#1}}
\newcommand{\thmref}[1]{Theorem~\ref{#1}}
\newcommand{\lemref}[1]{Lemma~\ref{#1}}
\newcommand{\algref}[1]{Algorithm~\ref{#1}}

\newcommand{\appref}[1]{Appendix~\ref{#1}}

\DeclareMathOperator*{\argmin}{arg\,min}

\newcommand{\Reg}{\textsc{R}_\textnormal{T}}

\newcommand{\Hv}{ \mathbf{V} }
\newcommand{\Hvs}{ \mathbf{\tilde V} }
\newcommand{\F}{ \mathbf{F} }
\newcommand{\Fs}{ \mathbf{\tilde F} }

\newcommand{\eps}{\varepsilon}
\renewcommand{\eps}{\epsilon}

\newcommand{\tnabla}{\widetilde{\nabla}}

\newcommand{\Expect}[1]{\mathbb E \left[\,#1\,\right]} 
\newcommand{\norm}[1]{\lVert {#1} \rVert} 
\newcommand{\ip}[2]{\langle{#1},{#2}\rangle} 

\newcommand{\proj}[1]{\Pi_{\mathcal X} \left( {#1} \right)}

\newcommand{\br}[1]{\left({#1}\right)} 
\newcommand{\bc}[1]{\left\{{#1}\right\}} 
\newcommand{\bs}[1]{\left[{#1}\right]} 

\newcommand{\ignore}[1]{}

\usepackage{dsfont}		

\usepackage{acronym}		

\usepackage[labelfont={bf,small},labelsep=colon,font=small]{caption}	
\usepackage[dvipsnames,svgnames]{xcolor}		
\colorlet{MyRed}{Crimson!90!Black}
\colorlet{MyBlue}{MediumBlue!90!Black}
\colorlet{MyGreen}{DarkGreen!80!Black}

\usepackage{pifont}		
\newcommand{\cmark}{\text{\ding{51}}}		
\newcommand{\xmark}{\text{\ding{55}}}		

\usepackage{subcaption}		
\usepackage{tikz}		
\usetikzlibrary{calc,patterns,positioning}

\usepackage{booktabs}		
\usepackage[inline,shortlabels]{enumitem}		
\setenumerate{itemsep=0pt,topsep=1pt,left=0pt}
\setitemize{itemsep=0pt,topsep=1pt,left=0pt}

\usepackage{xspace}		

\usepackage[numbers,sort&compress]{natbib}		

\bibpunct[, ]{[}{]}{,}{n}{,}{,}

\usepackage{hyperref}
\hypersetup{
colorlinks=true,
linktocpage=true,
pdfstartview=FitH,
breaklinks=true,
pdfpagemode=UseNone,
pageanchor=true,
pdfpagemode=UseOutlines,
plainpages=false,
bookmarksnumbered,
bookmarksopen=false,
bookmarksopenlevel=1,
hypertexnames=true,
pdfhighlight=/O,
urlcolor=MyBlue,linkcolor=MyBlue,citecolor=MyBlue,	
pdftitle={},
pdfauthor={},
pdfsubject={},
pdfkeywords={},
pdfcreator={pdfLaTeX},
pdfproducer={LaTeX with hyperref}
}

\usepackage[sort&compress,capitalize,nameinlink]{cleveref}		
\crefname{algorithm}{Algorithm}{Algorithms}
\crefname{equation}{Eq.}{Eqs.}

\usepackage{thmtools}		
\usepackage{thm-restate}		

\theoremstyle{plain}

\newtheorem{theorem}{Theorem}[section]
\newtheorem{lemma}{Lemma}[section]

\newtheorem{assumption}{Assumption}[section]
\newtheorem{proposition}{Proposition}[section]
\newtheorem*{corollary*}{Corollary}		

\theoremstyle{definition}

\newtheorem*{definition*}{Definition}		
\newtheorem*{assumption*}{Assumptions}		
\newtheorem*{example*}{Example}		

\theoremstyle{remark}
\newtheorem{remark}{Remark}	[section]	
\newtheorem*{remark*}{Remark}		

\newcommand{\debug}[1]{#1}		

\newcommand{\newmacro}[2]{\newcommand{#1}{\debug{#2}}}		
\newcommand{\newop}[2]{\DeclareMathOperator{#1}{\debug{#2}}}		

\DeclarePairedDelimiterX{\setdef}[2]{\{}{\}}{#1:#2}		
\DeclarePairedDelimiterXPP{\exclude}[1]{\mathopen{}\setminus}{\{}{\}}{}{#1}

\newcommand{\R}{\mathbb{R}}		

\DeclareMathOperator{\bigoh}{\mathcal{O}}		
\DeclarePairedDelimiterXPP{\bigof}[1]{\bigoh}{(}{)}{}{#1}		
\DeclareMathOperator{\dist}{dist}		
\DeclareMathOperator{\relint}{ri}		

\newcommand{\ie}{i.e.,\xspace}		

\newcommand{\alt}[1]{#1'}		
\newcommand{\altalt}[1]{#1''}		

\newmacro{\dd}{\:d}		

\newmacro{\const}{C}		
\newmacro{\constalt}{B}		
\newmacro{\constdual}{\const_{\texttt{dual}}}		
\newmacro{\conststoch}{\constalt}		

\newmacro{\multi}{\texttt{H}}		

\newmacro{\coef}{\lambda}		
\newmacro{\param}{\theta}		
\newmacro{\params}{\Theta}		

\newmacro{\pexp}{r}		
\newmacro{\qexp}{q}		
\newmacro{\rexp}{r}		

\newmacro{\beforestart}{0}		
\newmacro{\start}{1}		
\newmacro{\afterstart}{2}		
\newmacro{\running}{\start,\afterstart,\dotsc\,}		

\newmacro{\run}{t}		
\newmacro{\runprev}{\run-1}		
\newmacro{\runalt}{s}		
\newmacro{\runaltalt}{\alt \run}		
\newmacro{\nRuns}{T}		
\newmacro{\runs}{\mathcal{\nRuns}}		

\newmacro{\state}{X}		
\newmacro{\stateavg}{\bar\state}		
\newmacro{\stateopt}{\tilde\state}		
\newmacro{\statealt}{\alt\state}		
\newmacro{\query}{x}		
\newmacro{\out}{\bar\query}		

\newcommand{\beforeinit}[1][\state]{\debug{#1}_{\beforestart}}		
\newcommand{\init}[1][\state]{\debug{#1}_{\start}}		

\newcommand{\iter}[1][\state]{\debug{#1}_{\runalt}}		
\newcommand{\iterlead}[1][\state]{\debug{#1}_{\runalt+\frac{1}{2}}}		

\newcommand{\prev}[1][\state]{\debug{#1}_{\run-1}}		
\newcommand{\curr}[1][\state]{\debug{#1}_{\run}}		
\renewcommand{\next}[1][\state]{\debug{#1}_{\run+1}}		

\newcommand{\beforelead}[1][\state]{\debug{#1}_{\run-\frac{1}{2}}}		
\newcommand{\lead}[1][\state]{\debug{#1}_{\run+\frac{1}{2}}}		

\newcommand{\last}[1][\state]{\debug{#1}_{\nRuns}}		
\newcommand{\lastlead}[1][\state]{\debug{#1}_{\nRuns+\frac{1}{2}}}		

\newop{\Eq}{Eq}		
\newop{\Nash}{NE}		

\newop{\brep}{br}		
\newmacro{\reg}{\mathcal{R}}		
\newop{\regstoch}{\tilde{\reg}}		
\newop{\preg}{Reg}		

\newop{\val}{val}		
\newmacro{\play}{i}		
\newmacro{\playalt}{j}		
\newmacro{\playaltalt}{k}		
\newmacro{\nPlayers}{N}		
\newmacro{\players}{\mathcal{\nPlayers}}		

\newmacro{\pure}{a}		
\newmacro{\purealt}{\beta}		
\newmacro{\purealtalt}{\gamma}		
\newmacro{\nPures}{n}		
\newmacro{\pures}{\mathcal{A}}		

\newmacro{\strat}{p}		
\newmacro{\stratalt}{\alt\strat}		
\newmacro{\strataltalt}{\altalt\strat}		
\newmacro{\strats}{\mathcal{X}}		
\newmacro{\intstrats}{\strats^{\circle}}		

\newmacro{\pay}{u}		
\newmacro{\payv}{v}		
\newmacro{\loss}{\ell}
\newmacro{\lossv}{\sample}
\newmacro{\pot}{\Phi}		
\newmacro{\meanpot}{\pot}		

\newmacro{\game}{\Gamma}		
\newmacro{\meangame}{\game}		
\newmacro{\gameall}{\game(\players,\points,\loss)}		

\newmacro{\fingame}{\Gamma}		
\newmacro{\fingameall}{\Gamma(\players,\pures,\pay)}		

\newmacro{\gmat}{g}		
\newmacro{\gdist}{\dist_{\gmat}}
\newmacro{\mfld}{M}		
\newmacro{\form}{\omega}		

\newmacro{\tvec}{z}		
\newmacro{\uvec}{u}		

\newmacro{\ball}{\mathbb{B}}		

\newmacro{\vecspace}{\mathcal{V}}		
\newmacro{\subspace}{\mathcal{W}}		

\newmacro{\bvec}{e}		
\newmacro{\bvecs}{\mathcal{E}}		

\newmacro{\coord}{i}		
\newmacro{\coordalt}{\coord^{\prime}}		
\newmacro{\coordaltalt}{\coordalt^{\prime}}		
\newmacro{\nCoords}{d}		
\newmacro{\dims}{\nCoords}		
\newmacro{\vdim}{\nCoords}		

\newmacro{\pspace}{\vecspace}		
\newmacro{\dspace}{\vecspace^{\ast}}		

\newmacro{\pstate}{z}		

\newmacro{\dstate}{Y}		
\newmacro{\dvec}{{v}}		
\newmacro{\disdstate}{\hat\dstate}		
\newcommand{\half}[2]{#1_{#2 + 1/2}}		

\newmacro{\Anc}{\textsf{r}}
\newmacro{\anchor}{R}
\newmacro{\ptest}{\tilde{\pstate}}		
\newmacro{\test}{\tilde{\state}}		
\newmacro{\dtest}{\tilde{\drecom}}		
\newmacro{\testsignal}{\tilde{\signal}}		
\newmacro{\precom}{\pstate}	
\newmacro{\drecom}{\dstate}		

\newmacro{\dispstate}{Z}		
\newmacro{\dispstatealt}{\bar\dispstate}		

\newmacro{\dpoint}{y}		

\newmacro{\dpointave}{\bar\dpoint}		
\newmacro{\dpointalt}{W}		
\newmacro{\dpointaltalt}{\altalt\dpoint}		
\newmacro{\dpoints}{\mathcal{Y}}		

\newmacro{\sgrad}{g}
\newmacro{\shess}{H}
\newmacro{\vargrad}{\sigma_g}
\newmacro{\varhess}{\sigma_H}

\newmacro{\wgrad}{a}                
\newmacro{\wgradsum}{A}
\newmacro{\wavg}{b}                 
\newmacro{\wavgsum}{B}

\newmacro{\eucdiam}{D}

\newmacro{\mat}{\mathbf{M}}		
\newmacro{\hmat}{\mathbf{H}}		

\newmacro{\ones}{\mathbf{1}}		
\newmacro{\eye}{\mathbf{I}}		
\newmacro{\zer}{\mathbf{0}}		

\DeclarePairedDelimiterXPP{\dnormdef}[1]{}{\lVert}{\rVert}{_{\ast}}{#1}
\DeclarePairedDelimiterXPP{\dnorm}[1]{}{\lVert}{\rVert}{_{\ast}}{#1}

\DeclarePairedDelimiterXPP{\onenorm}[1]{}{\lVert}{\rVert}{_{1}}{#1}		
\DeclarePairedDelimiterXPP{\twonorm}[1]{}{\lVert}{\rVert}{_{2}}{#1}		
\DeclarePairedDelimiterXPP{\supnorm}[1]{}{\lVert}{\rVert}{_{\infty}}{#1}		

\DeclarePairedDelimiterXPP{\altdnorm}[1]{}{\lVert}{\rVert'}{_{\ast}}{#1}

\DeclarePairedDelimiterX{\braket}[2]{\langle}{\rangle}{#1\mathopen{},\mathopen{}#2}

\DeclarePairedDelimiterX{\inner}[2]{\langle}{\rangle}{#1,#2}		
\newmacro{\cartprod}{\bigtimes}

\newop{\Opt}{Opt}		
\newop{\Sol}{Sol}		
\newop{\orcl}{\mathsf{G}}		

\newmacro{\obj}{f}		
\newmacro{\objalt}{\alt \obj}		
\newmacro{\sobj}{F}		
\newmacro{\func}{\textsl{g}}

\newmacro{\gvec}{g}		
\newmacro{\oper}{A}		
\newmacro{\vecfield}{v}		
\newmacro{\seed}{\sample}		

\newcommand{\sol}[1][\point]{#1^{\ast}}		

\newmacro{\gbound}{G}		
\newmacro{\lips}{L}

\newmacro{\strong}{\kappa}		
\newmacro{\smooth}{\lips}		

\newmacro{\cvx}{\mathcal{K}}		
\newmacro{\compact}{\mathcal{X}}		
\newmacro{\subd}{\partial}		
\newmacro{\subsel}{\nabla}		

\newmacro{\minmax}{L}		

\newmacro{\minvar}{\theta}		
\newmacro{\minvaralt}{\alt\minvar}		
\newmacro{\minvars}{\Theta}		

\newmacro{\maxvar}{\phi}		
\newmacro{\maxvaralt}{\alt\maxvar}		
\newmacro{\maxvars}{\Phi}		

\newmacro{\hreg}{h}		
\newmacro{\proxdom}{\points_{\hreg}}		

\newmacro{\breg}{D}		
\newmacro{\mprox}{P}		

\newmacro{\fench}{F}		
\newmacro{\mirror}{Q}		

\newmacro{\hstr}{K_{\hreg}}		
\newmacro{\bregdiam}{B_{\hreg}}		
\newmacro{\range}{R_{\hreg}}		

\DeclarePairedDelimiterXPP{\proxof}[2]{\mprox_{#1}}{(}{)}{}{#2}		
\DeclarePairedDelimiterXPP{\bregof}[2]{\breg}{(}{)}{}{#1,#2}		
\DeclarePairedDelimiterXPP{\fenchof}[2]{\fench}{(}{)}{}{#1,#2}		

\newmacro{\zone}{\mathbb{D}}		

\newop{\Eucl}{\Pi}		
\newop{\logit}{\Lambda}		
\newmacro{\subgrad}{W}

\newmacro{\point}{x}		
\newmacro{\pointalt}{\alt\point}		
\newmacro{\pointaltalt}{\altalt\point}		
\newmacro{\points}{\mathcal{X}}		
\newmacro{\intpoints}{\relint\points}		

\newmacro{\base}{\point^{\ast}}		
\newmacro{\basealt}{u^{\ast}}		

\newmacro{\real}{x}
\newmacro{\realalt}{\alt\real}
\newmacro{\realaltalt}{\altalt \real}

\newmacro{\open}{\mathcal{U}}		
\newmacro{\closed}{\mathcal{C}}		
\newmacro{\cpt}{\mathcal{K}}		
\newmacro{\nhd}{\mathcal{U}}		

\newop{\ex}{\mathbb{E}}		
\newop{\prob}{\mathbb{P}}		
\newop{\Var}{Var}		
\newop{\simplex}{\Delta}		

\DeclarePairedDelimiterXPP{\exof}[1]{\ex}{[}{]}{}{
	 #1}

\DeclarePairedDelimiterXPP{\probof}[1]{\prob}{(}{)}{}{
	 #1}

\newmacro{\sample}{\omega}		
\newmacro{\samples}{\Omega}		

\newmacro{\filter}{\mathcal{F}}		
\newmacro{\probspace}{(\samples,\filter,\prob)}		

\newmacro{\event}{E}       
\newmacro{\eventalt}{H}       
\newmacro{\mean}{\mu}		
\newmacro{\sdev}{\sigma}		
\newmacro{\variance}{\sdev^{2}}		

\newmacro{\step}{\gamma}    
\newmacro{\stepalt}{\gamma}		
\newmacro{\stepaltalt}{\lambda}		
\newmacro{\stepada}{\theta}		
\newmacro{\stepscale}{\beta_0}

\newmacro{\learn}{\eta}		
\newmacro{\dstep}{\psi}		

\newmacro{\proper}{\tau}		

\newmacro{\signal}{\gvec}		
\newmacro{\altsignal}{\bar\signal}		
\newmacro{\error}{Z}		
\newmacro{\noise}{U}		
\newmacro{\bias}{b}		
\newmacro{\brown}{W}		

\newmacro{\serror}{\theta}		
\newmacro{\snoise}{\xi}		
\newmacro{\sbias}{\psi}		

\newmacro{\sbound}{M}		
\newmacro{\bbound}{B}		
\newmacro{\noisepar}{\sdev}		
\newmacro{\noisevar}{\texttt{var}}		

\newcommand{\gd}{\debug{\textsc{Gd}}\xspace}
\newcommand{\sgd}{\debug{\textsc{Sgd}}\xspace}
\newcommand{\newton}{\debug{\textsc{Newton's}}\xspace}
\newcommand{\method}{\debug{\textsc{Extra-Newton}}\xspace}
\newcommand{\adagrad}{\debug{\textsc{AdaGrad}}\xspace}
\newcommand{\accelegrad}{\debug{\textsc{AcceleGrad}}\xspace}
\newcommand{\unixgrad}{\debug{\textsc{UnixGrad}}\xspace}

\newcommand{\fenchelgame}{\debug{\textsc{FenchelGame}}\xspace}

\newmacro{\cost}{c}

\newmacro{\boundcoord}{a}  
\newmacro{\smoothcoord}{b}		

\newmacro{\noiseave}{\tilde{\noise}}  
\newmacro{\noisetest}{\tilde{\noise}}  

\newmacro{\mindiff}{\rho}
\newmacro{\diff}{H}

\newmacro{\resource}{s}
\newmacro{\nResources}{d}
\newmacro{\resources}{\mathcal{S}}
\newmacro{\inflow}{\rho}
\newmacro{\load}{x}

\newmacro{\spectron}{\mathcal{D}}
\newmacro{\nInput}{M}
\newmacro{\nOutput}{N}
\newmacro{\chanmat}{\mathbf{H}}
\newmacro{\covmat}{\mathbf{X}}
\newmacro{\bx}{\mathbf{x}}
\newmacro{\by}{\mathbf{y}}
\newmacro{\bz}{\mathbf{z}}
\newmacro{\capa}{R}
\newmacro{\power}{P}

\newmacro{\maxsel}{m}

\newmacro{\shape}{\chi}
\newmacro{\diampoints}{\norm{\points}}

\newmacro{\underconst}{\const}
\newmacro{\slow}{\mathrm{slow}}
\newmacro{\fast}{\mathrm{fast}}

\newcommand{\indep}{\perp\!\!\!\!\perp} 

\begin{document}

\title
[Extra-Newton]
{Extra-Newton: A First Approach to Noise-Adaptive Accelerated Second-Order Methods}

\author
[K.~Antonakopoulos]
{Kimon Antonakopoulos$^{\boldsymbol{\sharp}, \dag}$}
\address{$^{\boldsymbol{\sharp}}$\,%
EQUAL CONTRIBUTION.}
\email{kimon.antonakopoulos@epfl.ch}

\author
[A.~Kavis]
{Ali Kavis$^{\boldsymbol{\sharp}, \dag}$}
\address{$^{\dag}$\,%
Laboratory for Information and Inference Systems, IEM STI EPFL, Lausanne, Switzerland.}
\email{ali.kavis@epfl.ch}

\author
[V.~Cevher]
{Volkan Cevher$^{\dag}$}
\email{volkan.cevher@epfl.ch}

\subjclass[2020]{Primary 90C25, 90C15; secondary 68Q32, 68T05.}
\keywords{%
Universal methods;
dimension-freeness;
dual extrapolation;
rate interpolation.}

\thanks{}

\etocdepthtag.toc{mtchapter}
\etocsettagdepth{mtchapter}{subsection}
\etocsettagdepth{mtappendix}{none}

\begin{abstract}
	This work proposes a universal and adaptive second-order method for minimizing second-order smooth, convex functions. Our algorithm achieves $O(\sigma / \sqrt{\nRuns})$ convergence when the oracle feedback is stochastic with variance $\sigma^2$, and improves its convergence to $O( 1 / \nRuns^3)$ with deterministic oracles, where $T$ is the number of iterations. Our method also interpolates these rates without knowing the nature of the oracle apriori, which is enabled by a parameter-free adaptive step-size that is oblivious to the knowledge of smoothness modulus, variance bounds and the diameter of the constrained set. To our knowledge, this is the first universal algorithm with such global guarantees within the second-order optimization literature. 
\end{abstract}

\maketitle
\allowdisplaybreaks		
\acresetall		

\section{Introduction} \label{sec:introduction}

Over the last few decades, first-order (convex) minimization methods have gained popularity for modern machine learning and optimization problems due to their efficient per-iteration cost and \textit{global convergence} properties. 
The literature on first-order methods is rather dense and extensive with a concrete, thorough understanding of the optimal \textit{global} convergence behavior.
Focusing on the more relevant settings of smooth, convex minimization, the lower bounds have been well-established; $O(\sigma / \sqrt{\nRuns})$  when the gradient feedback is noisy with variance $\sigma^2$, and $O(1 / \nRuns^2)$ under deterministic first-order oracles \citep{nemirovskii1983problem, nesterov2003introductory}. 
Under slight variations of the aforementioned problem setting, there exists an extensive amount of work that enjoys the latter, ``accelerated'' rate \citep{nesterov1983acceleration, nesterov1988approach, nesterov2005smooth, tseng2008accelerated, xiao2010dual, lan2012optimal, allenzhu2016linear, levy2018online, wang2018acceleration, diakonikolas2018accelerated, cutkosky2019anytime, kavis2019universal, joulani2020simpler, AVCL+22, liu2022convergence}. 

On the contrary to its first-order analogue, the literature on \textit{global convergence} of {second-order}, smooth methods is notably sparse with many open questions standing even in the simplest problem formulations. 
Following the pioneering works of \citet{bennett1916newton, kantorovich1948functional}, Newton's method and its variations \citep{levenberg1944method, marquardt1963algorithm} are considered as the staple of second-order methods in optimization. 
Although its powerful local convergence behavior has been repeatedly demonstrated \citep{conn2000trust, mishchenko2019stochastic}, studies on its global behavior are relatively limited.
Prior attempts at tackling global convergence mostly make additional structural assumptions on the objective function \citep{polyak2006newton, marteau2019globally, mishchenko2019stochastic} or assume extra regularity conditions on the Hessian \citep{karimireddy2018global} beyond the simplest smooth and convex setting.
Over the last decade, we have witnessed important progress towards a more complete theory of globally-convergent second-order methods (more on this shortly), and yet there remains many important questions unanswered, which we will delve into in this paper.

To motivate the perspective in our technical endeavour, we take a small detour to introduce the idea of \textit{universality}, which we particularly characterize as \textit{adaptation to the level of noise in oracle feedback}. 
Enabled by the recent advances in online optimization, universal first-order algorithms essentially attain the $O(\sigma / \sqrt{\nRuns} + 1 / \nRuns^2)$ convergence for convex minimization problems, interpolating between stochastic and deterministic rates. 
There exist a plethora of algorithms that enjoy this rate under different sets of assumptions for both minimization scenarios (for convex and non-convex settings, we refer the reader to \citep{lan2012optimal, kavis2019universal, ene2021adaptive, joulani2020simpler,AVCL+22} and  \citep{ward2019adagrad, levy2021storm,kavis2022high, liu21}, respectively), and the more general framework of variational inequalities \citep{BL19, APKM+21, VAM21, HAM21, ABM19, HAM22, antonakopoulos2021adaptive}.
However, we observe that such universal results do not exist in second-order literature, hence, it is only natural to ask,

\begin{center}
    \textit{Can we design a simple second-order method that will achieve\\accelerated universal rates beyond $O(\sigma / \sqrt{\nRuns} + 1 / \nRuns^2)$?}
\end{center}

More recently, global sub-linear convergence rates for second-order methods have been characterized by \citep{nesterov2006cubic} for second-order smooth and convex setting. Essentially, the so-called Cubic Regularized Newton's Method combines the quadratic Taylor approximation in the typical Newton update with a cubic regularization term.
At the expense of solving a cubic problem, this method achieves $O(1 / \nRuns^2)$ convergence rate. Shortly after, \citet{nesterov2008accelerating} proposes an accelerated version of the cubic regularization idea with $O(1/\nRuns^3)$ value convergence, pioneering a new direction of research in the study of globally-convergent second-order methods \citep{mishchenko2021regularized}. This idea has been studied further for different settings in convex optimization \citep{jiang2017unified, jiang2020unified} with the same accelerated $O(1/\nRuns^3)$ rate and extended to non-convex realm \citep{cartis2011adaptive1, cartis2011adaptive2}, obtaining the analogous rates of $O(1/\nRuns^{2/3})$ and $O(1/\nRuns^{1/3})$ for finding first-order and second-order stationary points, respectively, leading the way for further investigations \citep{bellavia2020adaptive, dussault2021scalable, chen2022accelerating}.

Notice that accelerated cubic regularization is \textit{sub}-optimal such that recent studies prove a respective lower-bound for second-order smooth, convex problems as $O(1/\nRuns^{7/2})$~\citep{agarwal2018lower, arjevani2019oracle}. The first line of research that shrinks the gap between the upper and lower bounds for achieving an \textit{almost}-optimal (more on this shortly) convergence \citep{nesterov2018lectures} is the so-called ``bisection-type'' methods. Pioneered by \citet{monteiro2013accelerated}, these class of algorithms propose a conceptual method where the step-size of the algorithm \textit{implicitly} depends on the next iterate. To resolve, the authors propose a bisection procedure that simultaneously finds a step-size/next iterate pair that satisfies the conditions of the iterative update, which enables the convergence rate of $O(1/\nRuns^{7/2})$, modulo the complexity of bisection procedure. This idea was very recently generalized for higher-order tensor methods \citep{gasnikov2019optimal}. Not so surprisingly, the same construction finds application in variational inequality (VI) and min-max optimization literature \citep{bullins2020higher, jiang2022generalized}. Very recently and concurrently to our work, \cite{carmon2022optimal} propose the first bisection free acceleration for second-order methods, that achieves the optimal $O(1/\nRuns^{7/2})$. The authors define an \emph{explicit}, deterministic procedure called MS oracle and compute the step-size using a standard line-search procedure enabling them to achieve optimal rates while adaptively computing the step-size without needing to know the smoothness constant.

Although there are promising results with an increasing interest into second-order --and also higher-order-- methods, we identify three main shortcomings in the literature, which we will systematically address in the sequel. First, bisection-type methods achieve the optimal convergence rate however, the search procedure is computationally very prohibitive \citep{nesterov2018lectures, lin2022perseus} and the resulting algorithms are complicated with many interconnected components. 
On the other hand, cubic regularization-based ideas propose a simple construction that achieves acceleration beyond $O(1/\nRuns^2)$ however, similar to previous methods, they either require the knowledge of smoothness constant or need to execute a standard line-search procedure to estimate it locally. A common drawback for both approaches is that the algorithmic constructions are designed for handling \textit{only} deterministic oracles and it is an open question whether such frameworks could immediately accommodate stochastic first and second-order information.
\paragraph{Our contributions:}
To address the aforementioned issues, we developed the first universal and adaptive second-order algorithm, \method, for convex minimization. We summarize our contributions as follows:
\begin{enumerate}
    \item
    We prove \method achieves the global convergence rate of $O(\frac{\sigma_g}{\sqrt{\nRuns}} + \frac{\sigma_H}{\nRuns^{3/2}} + \frac{\smooth \eucdiam^3}{\nRuns^{3}})$ that adapts simultaneously to the variance in the gradient oracle ($\sigma_g$) and Hessian oracle ($\sigma_H$) achieving the first universal convergence result in the literature.
    \item
    Our method is completely oblivious to any problem-dependent parameters including smoothness modulus, variance bounds on stochastic oracles, diameter of the constraint set and any possible bounds on the gradient and Hessian.
    \item
    We design the first adaptive step-size, in the sense of \citep{duchi2011adaptive, rakhlin2013optimization}, that successfully incorporates second-order information ``on-the-fly''.
    While doing so, we bypass any bisection or linesearch procedure, and propose a simple, intuitive algorithmic framework.
\end{enumerate}
From a technical point of view, what will allow us to achieve these results is the combination of three principal ingredients:
\begin{enumerate*}
[(\itshape i\hspace*{1pt})]
\item
proposing appropriate adjustments to Extra-Gradient \citep{korpelevich1976extragradient} that was originally designed for solving variational inequalities and min/max problems;
\item
an ``optimistic'' weighted iterate averaging scheme accompanied by an appropriate gradient rescaling strategy in the spirit of \citep{wang2018acceleration, diakonikolas2018accelerated, kavis2019universal} which allows us to obtain an accelerated rate of convergence by means of a generalized online-to-batch conversion (\thmref{thm:conversion}),
and
\item the glue that holds these elements together is an adaptive learning rate inspired by \citep{rakhlin2013optimization, kavis2019universal, antonakopoulos2021adaptive} which automatically rescales aggregated gradients and second order information. 
\end{enumerate*} 
In what follows, we shall explicate these arguments.
\begin{table}
\label{tbl:survey}
\centering
\caption{A survey on first and second-order algorithms with key properties}
\resizebox{\textwidth}{!}{\begin{tabular}{l|c|c|c|c|c|c|c}
	&\begin{tabular}{@{}c@{}} \textbf{AGD} \\ \citep{nesterov1983acceleration} \end{tabular}
	&\begin{tabular}{@{}c@{}} \textbf{UniXGrad} \\ \citep{kavis2019universal} \end{tabular}
	&\begin{tabular}{@{}c@{}} \textbf{Reg.} \\ \textbf{Newton} \\ \citep{mishchenko2021regularized} \end{tabular}
	&\begin{tabular}{@{}c@{}} \textbf{Accel.} \\ \textbf{Cubic Reg.} \\ \citep{nesterov2008accelerating}\end{tabular}
	&\begin{tabular}{@{}c@{}} \textbf{ANPE\tablefootnote{Note that the bisection procedure is computationally prohibitive, we defer the reader to \citep{nesterov2018lectures}, p.304-305.}} \\ \citep{monteiro2013accelerated}\end{tabular}
	&\begin{tabular}{@{}c@{}} \textbf{OptMS} \\ \citep{carmon2022optimal} \end{tabular}
	&\begin{tabular}{@{}c@{}} \textbf{Extra} \\ \textbf{Newton} \\ \textbf{{[}ours{]}}\end{tabular}
	\\
\hline
\textit{Rate}
	&$\frac{1}{\nRuns^2}$
	&$\frac{\sigma_g}{\sqrt{\nRuns}} + \frac{1}{\nRuns^2}$
	&$\frac{1}{\nRuns^2}$
	&$\frac{1}{\nRuns^3}$
	&$\frac{1}{\nRuns^{7/2}}$
	&$\frac{1}{\nRuns^{7/2}}$
	&$\frac{\sigma_g}{\sqrt{\nRuns}} + \frac{\sigma_H}{\nRuns^{3/2}} + \frac{1}{\nRuns^3}$
	\\
\hline
\textit{Bisection}
	&\color{red}\xmark
	&\color{red}\xmark
	&\color{red}\xmark
	&\color{red}\xmark
	&\color{Green}\cmark
	&\color{red}\xmark
	&\color{red}\xmark
	\\
\hline
\textit{Adapts to $L$}
	&\color{red}\xmark
	&\color{Green}\cmark
	&\color{red}\xmark
	& Partial
	&\color{red}\xmark
	&\color{Green}\cmark
	&\color{Green}\cmark
	\\
\hline
\textit{Noise-adaptive}
	&\color{red}\xmark
	&\color{Green}\cmark
	&\color{red}\xmark
	&\color{red}\xmark
	&\color{red}\xmark
	&\color{red}\xmark
	&\color{Green}\cmark
	\\
\hline
\end{tabular}}
\end{table}

\section{Problem setup} \label{sec:problem-setup}

Throughout the sequel, we will be focusing on solving (constrained) convex minimization problems of the general form:
\begin{equation}
\label{eq:opt}
\tag{Opt}
\begin{aligned}
\textrm{minimize}
	&\quad
	f(\point)
	\\
\textrm{subject to}
	&\quad
\point \in \compact.
\end{aligned}
\end{equation}
Formally, in the above $\points$ is a convex and compact subset of a $d$- dimensional normed space $\mathcal{V}\cong \mathbb{R}^{d}$ with diameter $D = \max_{x, y \in \compact} \norm{x - y}$, and $f:\mathcal{V}\to \mathbb{R}\cup\{ +\infty\}$ is a proper, lower semi-continuous, convex function with $\text{dom} f=\{x \in \R^{\vdim}:f(x)<+\infty\}\subset \compact$. 
To that end, we make a set of blanket assumptions for \eqref{eq:opt}. Following the vast literature of constrained convex minimization \citep{nesterov2006cubicreg, beck2009fast}, we consider ``simple'' constraint sets, i.e.,
\begin{assumption}
The constraint set $\compact$ of \eqref{eq:opt} possesses favorable geometry which facilitates a tractable projection operator.
\end{assumption}
In order to avoid trivialities, we also assume that the said problem admits at least a solution, i.e.
\begin{assumption}
The solution set $\compact^{\ast}=\argmin_{\point \in \compact}f(\point)$ of \eqref{eq:opt} is non-empty.
\end{assumption}
Furthermore, we assume that there exists a Lipschitz continuous selection $\point \mapsto \nabla^{2}\obj(\point)\in \R^{\vdim\times \vdim}$, \ie
\begin{equation}
\tag{H-smooth}
\label{eq:Hess-smooth}
    \|\nabla^{2}f(\point)-\nabla^{2}f(\pointalt)\| \leq \smooth \|\point-\pointalt\| \;\;\forall \point,\pointalt \in \points
\end{equation}
and in addition it satisfies the second order approximation:
\begin{equation}
\tag{Taylor}
\label{eq:Taylor}
    \obj(\point)=\obj(\pointalt)+\braket{\nabla \obj(\pointalt)}{\point-\pointalt}+\frac{1}{2} \ip{\nabla^2 \obj (\pointalt) (\point - \pointalt)}{\point - \pointalt} + O \br{ \norm{\point - \pointalt}^3 }
\end{equation}
To that end, combining \eqref{eq:Hess-smooth} and \eqref{eq:Taylor} we readily get the following inequality:
\begin{equation} \label{eq:second-order-taylor}
    \norm{\nabla \obj(\point)-\nabla \obj(\pointalt)-\nabla^{2} \obj(\pointalt)(\point-\pointalt)}\leq \frac{\smooth}{2}\norm{\point-\pointalt}^{2}
\end{equation}
The above equivalences are well-established and hence we omit their proofs (we defer for a panoramic view to \citep{nesterov2021implementable})

\paragraph{Oracle feedback structure}
From an algorithmic point of view, we aim to solve \eqref{eq:opt} by using methods that require access to a (stochastic) first and second order-oracle. Before we move forward with the methodology, we shall introduce the definitions and notations for this oracle model which we will use in algorithm definitions and technical discussions. 
Let $g(\point, \xi)$ denote the stochastic gradient evaluated at $\point$ with randomness defined by $\xi$ and $H(\point, \xi)$ be the stochastic Hessian at $\point$ with $\xi$ describing the randomness of the oracle, such that
\begin{equation} \label{eq:stochastic-oracle}
\begin{aligned}
    \Expect{ \sgrad(\point, \xi) \mid \point } &= \nabla \obj(\point), \qquad\quad\,\, \Expect{\norm{\sgrad(\point, \xi) - \nabla \obj(\point)}^2 \mid \point } \leq \vargrad^2\\
    \Expect{ \shess(\point, \xi) \mid \point } &= \nabla^2 \obj(\point), \qquad \Expect{ \norm{\shess(\point, \xi) - \nabla^2 \obj(\point)}^2 \mid \point } \leq \varhess^2
\end{aligned}
\end{equation}
Due to space constraints, we will also define an operator that accommodates second-order information and its respective stochastic counterpart.
\begin{equation} \label{eq:second-order-oracle}
\begin{aligned} 
    \F(\point;\pointalt) &= \nabla \obj (\pointalt) + \frac{1}{2}\nabla^2 \obj(\pointalt) (\point - \pointalt)\\
    \Fs(\point;\pointalt, \xi) &= \sgrad (\pointalt, \xi) + \frac{1}{2}\shess(\pointalt, \xi) (\point - \pointalt)\
\end{aligned}
\end{equation}
where $\F$ is essentially the gradient (with respect to $\point$) of the second-order Taylor polynomial. By definition, the operator $\F$ satisfies the second-order smoothness property in Eq.~\eqref{eq:second-order-taylor}

\section{Method} \label{sec:method}

In this section, we shall establish our universal second-order framework. Our presentation evolves around three key components: choosing the appropriate algorithmic template with the key motivations behind it, solving implementability issues that commonly arise in higher-order methods and finally designing a universal algorithm that can handle deterministic and noisy oracle feedback simultaneously without having prior knowledge. 
Our point of departure is the popular Extra-Gradient (EG) template; originally introduced by \citet{korpelevich1976extragradient} and further developed in \citet{nemirovski2004prox},
\begin{equation} \label{eq:EG} \tag{EG}
    \begin{aligned}
        \lead &= \proj{\curr - \curr[\stepalt] \nabla f(x_t)}\\
        \next &= \proj{\curr - \curr[\stepalt] \nabla f(x_{t+1/2})},
    \end{aligned}
\end{equation}
where $\proj{x} = \arg \min_{z \in \mathcal X} \norm{x - z}^2$ is the standard Euclidean projection onto the set $\mathcal X$.
In terms of output, the candidate solution returned by \eqref{eq:EG} after $\nRuns$ iterations is the so-called ``ergodic average''
\begin{equation}
\last[\stateavg]
	= \frac
		{\sum_{\run=\start}^{\nRuns} \curr[\wavg] \lead}
		{\sum_{\run=\start}^{\nRuns} \curr[\wavg]}
\end{equation}
Then, taking $\curr[\wavg] = \curr[\stepalt]$ and assuming the method's step-size $\curr[\stepalt]$ is chosen appropriately, $\last[\stateavg]$ enjoys the following universal guarantee \citep{JNT11,RS13-NIPS}:
\begin{equation}
    \exof{\obj(\overline{\state}_{\run})-\obj(\sol)}=\bigoh\bigg(\frac{1}{\nRuns}+\frac{\sigma}{\sqrt{\nRuns}}\bigg)
\end{equation}
where $\sigma$ signifies the effect of the noisy feedback. 
However, as it becomes apparent, the vanilla \eqref{eq:EG} template is not capable of matching the iconic $1/\nRuns^{2}$ for the smooth deterministic case. 
It is well-established in the literature of smooth, convex minimization that iterate averaging (or momentum in the sense of \citet{nesterov1983acceleration}) is essential for matching the $O(1/T^2)$ lower bounds. In fact, plain uniform averaging is not sufficient; one needs to introduce new iterates with \textit{increasing} weights. Precisely, this is equivalent to computing an average by taking $\curr[\wavg] = O(\run)$. However, we cannot fully characterize the acceleration machinery without what we like to call ``gradient weighting''. On top of (weighted) iterate averaging, gradients must be multiplied by the \textit{same order of weights} to achieve acceleration, which is a recurring theme in the literature of accelerated and universal optimization \citep{tseng2008accelerated, xiao2010dual, lan2012optimal, allenzhu2016linear, levy2018online, wang2018acceleration, cutkosky2019anytime, kavis2019universal, joulani2020simpler}. 

Going back to discussion on \eqref{eq:EG}, \citet{wang2018acceleration} and \citet{kavis2019universal} provide useful insights into acceleration within the context of \eqref{eq:EG}. \citet{wang2018acceleration} identifies a 2-player game with a particular structure called \fenchelgame framework, which essentially reduces to minimizing a smooth, convex function when the players cooperate. By introducing an ``optimistic'' weighted iterate averaging along with a complementary gradient weighting strategy, the framework recovers different acceleration schemes of Nesterov \citep{nesterov1983acceleration, nesterov1988approach, nesterov2005smooth}. On a related front, \citet{diakonikolas2018accelerated} proposes the first acceleration of \eqref{eq:EG} by appropriately integrating the optimistic averaging idea \citep{wang2018acceleration} into the \eqref{eq:EG} template as follows:
\begin{equation}
\label{eq:averages}
\begin{aligned}
\curr[\stateopt]
	&= \frac
		{\curr[\wavg]\curr + \sum_{\runalt=\start}^{\run-1} \iter[\wavg] \iterlead}
		{\sum_{\runalt=\start}^{\run} \iter[\wavg]}, \qquad
\lead[\stateavg]
	= \frac
		{\sum_{\runalt=\start}^{\run} \iter[\wavg] \iterlead}
		{\sum_{\runalt=\start}^{\run} \iter[\wavg]}
\end{aligned}
\end{equation}
where $\curr[\wavg]=O(\run)$ is the ``iterate averaging'' parameter.
Later on, \citet{kavis2019universal} designs an adaptive, universal variant of accelerated Mirror-Prox following the same optimistic averaging idea as in Eq.~\eqref{eq:averages}.
All in all, it is a recurring theme among accelerated algorithms to adopt weighted iterate averaging ($\curr[\wavg] = O(\run)$) with proportionate gradient weighting, and not so surprisingly, prior work establishes clear connections between the degree of weighting and convergence rate. 
\citet{cutkosky2019anytime} designs a black-box reduction that accelerates a class of online algorithms and proves that the rate of convergence of the reduction is $O(1 / \sum_{\run = \start}^{\nRuns} \curr[\wavg])$ for $\curr[\wavg] \in [1, \run]$. 
In retrospect, we aim at answering the following question; 

\begin{center}
\textit{What algorithmic construction would enable acceleration beyond $O(1/\nRuns^2)$?} 
\end{center}

\subsection{Implicit algorithm} \label{sec:implicit-method}

We give a first affirmative answer to the above question by presenting our implicit accelerated algorithm which is constructed upon \eqref{eq:EG}, and establish its convergence properties.
Note that the implicitness of the scheme serves as a gentle introduction to the actual explicit second order acceleration, which shall follow. 
Formally, our scheme is given via the following recursion:
\begin{equation} \label{eq:implicit} \tag{Implicit}
    \begin{aligned}
       \lead &= \proj{\curr - \curr[\stepalt] \curr[\wgrad] \F(\lead[\stateavg]; \curr[\stateopt])} 
       \\
       &= \arg \min_{\point \in \compact} \curr[\wgrad] \ip{\nabla \obj (\curr[\stateopt]) + \frac{1}{2} \nabla^2 \obj (\curr[\stateopt]) (\lead[\stateavg] - \curr[\stateopt]) }{\point - \curr} + \frac{\norm{\point - \curr}^2}{2 \gamma_t}\\[2mm]
            \next &= \proj{\curr - \curr[\stepalt] \curr[\wgrad] \nabla \obj(\lead[\stateavg])}\\
            &= \arg \min_{\point \in \compact} \curr[\wgrad] \ip{\nabla \obj (\lead[\stateavg])}{\point - \curr} + \frac{\norm{\point - \curr}^2}{2 \gamma_t} 
    \end{aligned}
\end{equation}
with $\proj{\point}$ denoting the Euclidean projection of $\point$ onto $\compact$, average sequences $\curr[\stateopt]$ and $\lead[\stateavg]$ defined as in \eqref{eq:averages} and the adaptive step-size $\step_{\run}$ defined as (for some $\stepalt, \stepscale > 0$):
\begin{align} \label{eq:adaptive-stepsize-implicit}
    \curr[\stepalt] = \frac{\step}{\sqrt{ \stepscale + \sum_{\runalt=\start}^{\run-1} \iter[\wgrad]^2 \norm{\nabla \obj(\iterlead[\stateavg]) - \F(\iterlead[\stateavg]; \iter[\stateopt])}^2 }}.
\end{align}

The implicit nature of \eqref{eq:implicit} originates from $\state_{\run+1/2}$ update (which we shall refer to as (corrected) extrapolation step at times) since $\lead[\stateavg]$ depends upon $\lead$ itself. 
However, this scheme exhibits several key differences from the vanilla \eqref{eq:EG}, which constitute the fundamental parts of our second-order acceleration machinery. In particular, we have:
\begin{enumerate}[label=(\roman*)]
    \item integration of second-order updates for sharper extrapolation steps - first step of acceleration.
    \item interplay between averaging ($\curr[\wavg]$) and gradient weighting ($\curr[\wgrad]$) which allows more aggressive averaging - second step of acceleration.
    \item adaptive step-size in the sense of \citet{rakhlin2013optimization} - key to adaptivity and universality.
\end{enumerate}
\paragraph{Second-order updates:}
First, we will consider the particular interpretation of \eqref{eq:EG} as an approximation to the Proximal Point method~\citep{rockafellar1976monotone} which serves as motivation for the accommodation of second-order information in our scheme.
\begin{equation} \label{eq:PP} \tag{PP}
    \begin{aligned}
        \next &= \curr - \curr[\stepalt] \nabla \obj (\next).
    \end{aligned}
\end{equation}
In particular, \eqref{eq:EG} tries to approximate $\next$ by generating the extrapolated point $\lead$, and make use of the gradient at $\lead$ to take a step from $\curr$ to $\next$. Therefore, if the algorithm is able to compute a sharper estimate in the extrapolation step, it should be able live up to the fame of \eqref{eq:PP} and display faster convergence. To this end, we augment the extrapolation step by introducing second-order term. 
Essentially, our algorithm makes use of \textit{second-order Taylor approximation}, as opposed to first-order expansion, only for the extrapolation step, trading-off sharper approximation with second-order information. 

\paragraph{Iterate averaging and gradient weighting:}
Now, we turn our attention to the second component in our acceleration machinery; averaging and weighting. Recall that the acceleration framework of \citet{cutkosky2019anytime} guarantees a value convergence rate of $O(1/t^{p+1})$ when weighting factor satisfies $b_t = O(t^p)$ with $p \in [0,1]$. We take this result one step beyond in two fronts; our algorithm exploits higher-order smoothness in order to extend this bound for $p \in [0,2]$, implying the accelerated rate of $O(1/\nRuns^3)$. Second, we observe that previous work restricts the choice of gradient weights and averaging weights by taking $\curr[\wgrad] \approx \curr[\wavg]$. We decouple those weights by allowing the sequences $\curr[\wgrad]$ and $\curr[\wavg]$ to be \emph{different}, which in turn equips us with more aggressive iterate averaging when necessary.

\paragraph{Adaptive step-size:}
As the final component, we study the adaptive step-size \eqref{eq:adaptive-stepsize-implicit} from the parameter adaptation perspective (\ie adaptation to the Lipschitz modulus) and expand on its universal properties in the next section. 
The vast literature on adaptive methods predominantly rely on constructions of AdaGrad-like decreasing step-size policies by accumulating the observed gradient norms in its denominator. 
The intuition behind this choice is that whenever the method approaches a solution, the vanishing gradients bring about stabilization, ensuring progress around the solution's neighborhood.
However, this idea fails
for (compactly) constrained problems; when the solution lies on the boundary.
So inspired by \cite{rakhlin2013optimization}, we design a constraint-aware step-size by accumulating $\norm{ \nabla \obj(\lead[\stateavg]) - \F(\lead[\stateavg]; \curr[\stateopt]) }^2$ which converges to $0$ as $\lead[\stateavg] - \curr[\stateopt] \to 0$; which in turn implies convergence of the algorithm. To our knowledge, this is the first adaptive step-size that accommodates second order information.


Having established the core components of our design, we are in position to present the first accelerated convergence rate guarantee for \eqref{eq:implicit}. Formally, this is given by the following.

\begin{theorem} \label{thm:implicit}
    Let $\{ \lead \}_{\run = \start}^{\nRuns}$ be generated by \eqref{eq:implicit} run with the adaptive step-size policy \eqref{eq:adaptive-stepsize-implicit} where $\curr[\wgrad]=\run^{2}$, $\curr[\wavg]=\run^{p}$ with $p \geq 2$. Assume that $\obj$ satisfies \eqref{eq:Hess-smooth} then, it is ensured that:
    \begin{align*}
        \obj (\lastlead[\stateavg]) - \obj(\sol) \leq O \br{ \frac{ \max \bc{ \sqrt{\stepscale} \frac{\eucdiam^2}{\stepalt}, L \frac{\eucdiam^4 + \eucdiam \stepalt^3}{\stepalt} } }{\nRuns^3} }
    \end{align*}
    When $\stepalt = D$, we obtain the converge rate $O \br{ \frac{ \max \bc{\smooth \eucdiam^3, \sqrt{\stepscale} D }}{\nRuns^3}}$.
\end{theorem}

\begin{remark}
We emphasize that the above rate \emph{does not} require any prior knowledge of problem paramaters such as $\smooth$, $\eucdiam$, time-horizon $\nRuns$ and any bounds on gradient/Hessian norms. In order to have better dependence on $\eucdiam$ one could set $\stepalt = \eucdiam$, and our rate of $O(1/T^3)$ holds irrespective of $\stepalt$.
\end{remark}

\subsection{Explicit algorithm} \label{sec:explicit-method}

Despite the fact that \eqref{eq:implicit} improves upon the accelerated rate of $O(1/\nRuns^2)$, one may easily observe that it exhibits the following drawbacks: 
\begin{enumerate}
\item
\eqref{eq:implicit} is a conceptual algorithm and therefore, \emph{not} implementable in practice.
\item
A fortiori, it cannot provide rate interpolation guarantees as it does not have the machinery to simultaneously cope with deterministic and stochastic feedback. 
\end{enumerate}
As discussed earlier, a common strategy for overcoming this implicit construction 
is using a bisection/line-search procedure \citep{jiang2022generalized, monteiro2013accelerated, bullins2020higher}. 
Depending on the context, this procedure serves two \emph{distinct} purposes. 
Primarily, it tackles the implicit nature of the update rule by simultaneously finding a pair of ($\curr[\stepalt]$, $\lead$) and 
secondly, it enables adaptation to the second-order smoothness. 
However, one may identify major setbacks with these approaches; first, it is not clear how to handle stochastic oracles for executing the search procedure, so it is not capable of satisfying any universal guarantees.
Moreover, it yields a rather complicated procedure as a byproduct that has many moving parts.
To that end, we propose an alternative approach which not only yields a simple scheme, but also provides a universal algorithm that is able to handle noisy feedback on-the-fly. Without further ado, we display our explicit algorithm, \method, with appropriate modifications.
Having defined our main scheme, \algref{alg:explicit}, we will provide a more detailed description of its components. 
\begin{algorithm}[H]
	\caption{\method} \label{alg:explicit}
	\textbf{Input}: $\init[\state] \in \compact$, $\,\curr[\wgrad] = t^2$ and $A_t= \sum_{s=1}^{t} \iter[\wgrad]$, $\,\curr[\wavg] = t^p$ ($p\geq2$) and $B_t= \sum_{s=1}^{t} \iter[\wavg]$, $\,\gamma > 0$, $\xi_t \sim$ i.i.d.\\[-3mm]
	\begin{algorithmic}[1]
    	\FOR {$t=1$ to $T$}
                \smallskip
                \STATE $\quad\,\,\,\curr[\stepalt] = \dfrac{\step}{\sqrt{ \stepscale + \sum_{\runalt=\start}^{\run-1} \iter[\wgrad]^2 \norm{\sgrad(\iterlead[\stateavg], \xi_{s + \frac{1}{2}}) - \Fs(\iterlead[\stateavg]; \iter[\stateopt], \xi_s)}^2 }}$ \label{eq:adaptive-stepsize-explicit}
     		\smallskip
    		\STATE $\lead = \argmin_{\point \in \compact} \ip{ \curr[\wgrad] \sgrad (\curr[\stateopt], \xi_t) }{ \point } + \frac{\curr[\wgrad] \curr[\wavg] }{2 \curr[\wavgsum]} \ip{\shess ( \curr[\stateopt] , \xi_t) ( \point - \curr ) }{\point - \curr} + \frac{1}{2 \curr[\stepalt]} \norm{ \point - \curr }^2$
    		\STATE $\, \next= \argmin_{\point \in \compact} \ip{ \curr[\wgrad] \sgrad (\lead[\stateavg], \xi_{t+\frac{1}{2}}) }{ \point} + \frac{1}{2 \curr[\stepalt]} \norm{\point - \curr}^2$
    		\smallskip
    	\ENDFOR
    \end{algorithmic}
\end{algorithm}
\paragraph{Universal step-size} We modify our step-size (see Eq.~\eqref{eq:adaptive-stepsize-explicit}) in order to operate in the stochastic regime while making it noise-adaptive for rate interpolation. Using the same weighted averaging scheme in Eq.~\eqref{eq:averages}, we define the universal counterpart of the adaptive step-size,
Note that $\curr[\stepalt]$ is independent of any variable/randomness generated at iteration $\run$; it accumulates $\curr[\wgrad]^2 \norm{ \sgrad(\iterlead[\stateavg], \iterlead[\xi]) - \Fs(\iterlead[\stateavg]; \iter[\stateopt], \iter[\xi]) }^2$ up to $\run-1$. Therefore, the step-size is decoupled from the explicit update, \textit{a priori}.

Now, what remains is a new algorithmic design that will retain the accelerated convergence properties demonstrated by \eqref{eq:implicit} while having an explicit construction that is capable of automatically adjusting to noise level in the oracle feedback.
Before expanding upon the technical details of our strategy, let us take our time to explain the consequences of our explicit design compared to \eqref{eq:implicit}.

\paragraph{From implicit to explicit}
To obtain the explicit algorithm,
\begin{enumerate*}
[(\itshape i\hspace*{1pt})]
\item
we write the projection sub-problem in the $\argmin$ form;
\item
introduce \textit{stochastic} oracle feedback;
\item
for the second-order term, replace $\lead$ in $\lead[\stateavg]$ with the free variable $\point$;
then,
\item 
simplify as follows:
\end{enumerate*}
\begin{gather*}
    \frac{\curr[\wgrad]}{2} \langle \shess (\curr[\stateopt], \xi_t)( \lead[\stateavg] - \curr[\stateopt] ), \point - \curr \rangle\\
    \Downarrow\\
    \frac{\curr[\wgrad]}{2 } \Big\langle \shess (\curr[\stateopt], \xi_t) \Big( \frac{ \curr[\wavg] \lead + \sum_{\runalt=\start}^{\run-1} \iter[\wavg] \iterlead }{\curr[\wavgsum]} - \frac{ \curr[\wavg] \curr + \sum_{\runalt=\start}^{\run-1} \iter[\wavg] \iterlead }{\curr[\wavgsum]} \Big),\point - \curr \Big\rangle\\
    \Downarrow\\
    \frac{\curr[\wgrad]}{2 } \Big\langle \shess (\curr[\stateopt], \xi_t) \Big( \frac{ \curr[\wavg] \point + \sum_{\runalt=\start}^{\run-1} \iter[\wavg] \iterlead }{\curr[\wavgsum]} - \frac{ \curr[\wavg] \curr + \sum_{\runalt=\start}^{\run-1} \iter[\wavg] \iterlead }{\curr[\wavgsum]} \Big),\point - \curr \Big\rangle\\
    \Downarrow\\
    \frac{\curr[\wgrad] \curr[\wavg]}{2 \curr[\wavgsum]} \ip{ \shess (\curr[\stateopt], \xi_t)( \point - \curr ) }{ \point - \curr }
\end{gather*}
Given the bisection-type conceptual methods \citep{monteiro2013accelerated, jiang2022generalized, bullins2020higher}, it is surprising how smoothly we could transition from implicit to explicit \textit{once} we decouple the step-size from the current iteration \textit{apriori}. Moreover, the resulting update rule for the extrapolation step retains the quadratic structure as the $\next$ update rule. 
Having analyzed the components of the explicit scheme, we will first present the universal convergence rates then provide a concise explanation of the proof strategy with particular emphasis on the principal components of the analysis.
\begin{theorem} \label{thm:explicit}
     Let $\{ \lead \}_{\run = \start}^{\nRuns}$ be a sequence generated by \algref{alg:explicit}, run with the adaptive step-size policy \eqref{eq:adaptive-stepsize-explicit} and $\curr[\wgrad]=\run^{2},\curr[\wavg]=\run^{p}$ with $p \geq 2$. Assume that $\obj$ satisfies \eqref{eq:Hess-smooth}, and that Assumptions~\eqref{eq:stochastic-oracle} hold. Then, the following universal guarantee holds:
    \begin{align*}
        \obj (\lastlead[\stateavg]) - \obj (\sol[\point]) \leq O \br{ \frac{ { \frac{\eucdiam^2 + \stepalt^2}{\gamma} } \vargrad }{\sqrt{\nRuns}} + \frac{ { \frac{\eucdiam^3 + \eucdiam \stepalt^2}{\gamma} } \varhess }{\nRuns^{3/2}} + \frac{ \max \bc{ L {\frac{ \eucdiam^4 + \eucdiam \stepalt^3}{\gamma} }, \sqrt{\stepscale} { \frac{\eucdiam^2 + \stepalt^2}{\stepalt} } } }{\nRuns^3} }
    \end{align*}
    When $\stepalt = \eucdiam$, we obtain the target rate $O \br{ \frac{D \sigma_g}{\sqrt{\nRuns}} + \frac{D^2 \sigma_H}{\nRuns^{3/2}} + \frac{ \max \bc{\smooth \eucdiam^3, \sqrt{\stepscale} \eucdiam } }{\nRuns^3} }$.
\end{theorem}
\begin{remark}
    Similar to \thmref{thm:implicit}, \method achieves the preceding convergence rate independent of the knowledge of problem parameters. 
\end{remark}
Compatible with the \eqref{eq:EG}-based algorithmic design, our proof has the following main steps
\begin{enumerate}
    \item We perform an \emph{offline} regret analysis of Alg.~\ref{alg:explicit} and show adaptive regret bounds - see Prop.~\ref{prop:template-inequality}.
    \item We prove an anytime online-to-batch conversion framework, which generalizes that of \citet{cutkosky2019anytime}, through decoupling iterate averaging from gradient weighting - see \thmref{thm:conversion}.
    \item Combining the adaptive regret bound with the conversion theorem immediately implies \textit{universal, accelerated} value convergence of $O(\frac{D \sigma_g}{\sqrt{\nRuns}} + \frac{D^2 \sigma_H}{\nRuns^{3/2}} + \frac{ \max \bc{ \smooth \eucdiam^3, \sqrt{\stepscale} \eucdiam }}{T^3})$ - see \thmref{thm:explicit}.
\end{enumerate}
Let us begin with clarifying what \emph{offline regret} means for \algref{alg:explicit}. We define the (linear) regret considering the convention in both online learning \citep{rakhlin2013optimization, cutkosky2019anytime} and first-order acceleration literature \citep{wang2018acceleration, kavis2019universal, joulani2020simpler}. We measure the performance of our decisions for the extrapolation sequence such that after playing $\lead[\state]$, our algorithm observes and suffers the linear (weighted) loss with respect to $a_t \nabla f(\lead[\stateavg])$. Hence, we define the regret as
\begin{align} \label{eq:regret} \tag{Reg}
    \Reg(\point) = \sum_{\run = \start}^{\nRuns} a_t \ip{\nabla f(\lead[\stateavg])}{\lead[\state] - \point} 
\end{align}
where we run the algorithm for $T$ rounds. Next up, we provide our generalized conversion result.
\begin{theorem} \label{thm:conversion}
    Let $\Reg(\sol[\point])$ denote the anytime regret for the decision sequence $\{ \lead \}_{\run = \start}^{\nRuns}$ as in \eqref{eq:regret}, and define two sequences of non-decreasing weights $\curr[\wgrad]$ and  $\curr[\wavg]$ such that $\curr[\wgrad], \curr[\wavg] \geq 1$. As long as $ \curr[\wgrad] / \curr[\wavg]$ is ensured to be non-increasing, 
    \begin{align*}
        f(\last[\stateavg]) - f(\sol[\point]) \leq \frac{\Reg(\sol[\point])}{ \last[\wgrad] \frac{\last[\wavgsum]}{\last[\wavg]}}
    \end{align*}
\end{theorem}
\begin{remark}
    This conversion result holds independent of the order of smoothness of the objective as long as $\obj$ is convex. Moreover, it allows averaging parameter $\curr[\wavg]$ to be asymptotically larger than gradient weights $\curr[\wgrad]$, enabling a more aggressive averaging strategy when necessary. 
\end{remark}
To complement the lower bound to the regret $\Reg(\sol[\point])$, we present an upper bound that helps us explain how we exploit second-order smoothness for a more aggressive weighting, hence the rate $O(1 / \nRuns^3)$. 
\begin{proposition} \label{prop:template-inequality}
    Let $\{ \lead \}_{\run = \start}^{\nRuns}$ be generated by \algref{alg:explicit}, run with a non-increasing step-size sequence $\curr[\stepalt]$ and non-decreasing sequences of weights $\curr[\wgrad], \curr[\wavg] \geq 1$ such that $\curr[\wgrad] / \curr[\wavg]$ is also non-increasing. Then, the following guarantee holds:
    \begin{align*}
        \mathbb E \Reg(\sol[\point]) \leq \frac{1}{2} \Expect{ \frac{3\eucdiam^2}{\gamma_{T+1}} + \sum_{\run=\start}^\nRuns \next[\stepalt] \curr[\wgrad]^2 \norm{\sgrad(\lead[\stateavg], \lead[\xi]) - \Fs(\lead[\stateavg]; \curr[\stateopt], \curr[\xi])}^2 - \frac{\norm{\lead - \curr}^2}{\next[\stepalt]} }
    \end{align*}
\end{proposition}
Observe that the inequality in Proposition~\ref{prop:template-inequality} is agnostic to the design of our step-size in Eq.~\eqref{eq:adaptive-stepsize-explicit} as well as the selection of the weights as described in \thmref{thm:explicit}. 
It essentially applies to any non-increasing sequence of step-sizes and non-decreasing gradient weight sequence $\curr[\wgrad] \geq 1$. To obtain it, we neither used convexity nor the smoothness of the objective. In fact, the structure of the objective function, i.e., its convexity, will not be needed for upper-bounding the regret expression, and required only for the conversion in \thmref{thm:conversion}.

Now, let us explain how we make use of second-order smoothness for enjoying faster rates, and give a brief discussion of how the regret bound will look in its final form. First, we decompose the stochastic term $\norm{\sgrad(\lead[\stateavg], \lead[\xi]) - \Fs(\lead[\stateavg]; \curr[\stateopt], \curr[\xi])}^2$ into deterministic feedback and noise. Then, we argue that \textit{the noisy component} grows as $O(\sigma_H \nRuns^{3/2} + \sigma_g \nRuns^{5/2})$. On the other hand, achieving the accelerated $O(1/T^3)$ component of the universal rate amounts to showing that the regret has a constant, $O(1)$, component. In the worst-case sense, however, \textit{the deterministic component itself} grows as $O(\nRuns^{5/2})$. Fortunately, we identify that the negative term is ``large enough'' in magnitude to control the growth of the deterministic term, permitting a constant component $O(\smooth \eucdiam^2)$ for the regret.

Although the regret bound of $O(\smooth \eucdiam^3 + \eucdiam^2 \sigma_H \nRuns^{3/2} + \eucdiam \sigma_g \nRuns^{5/2})$ seems counter-intuitive from an online-learning perspective, it will make perfect sense when we discuss how second-order smoothness leads to ``faster'' conversion through more aggressive averaging. As a matter of fact, we will continue our discussion with how second-order smoothness helps us accelerate. It turns out that using \eqref{eq:Hess-smooth}, iterate averaging as in Eq.\eqref{eq:averages} and compactness of $\compact$, we can bound the negative term as,
\begin{align*}
    -\frac{1}{\next[\stepalt]} \norm{\lead - \curr}^2
    &\leq
    - \frac{1}{\smooth^2 \eucdiam^2 \next[\stepalt]} \run^4 \norm{ \nabla \obj(\lead[\stateavg]) - \F(\lead[\stateavg]; \curr[\stateopt]) }^2
\end{align*}
Observe that to seamlessly combine the positive and negative terms, our analysis enforces that $\curr[\wgrad] = O(\run^2)$ and $\curr[\wavg] = \Omega (\run^2)$. Then, the conversion implies a convergence rate of $\Reg(\sol[\point]) / \nRuns^3$, hence the recipe for acceleration. Therefore, the constant component of the regret amounts to $O(1/\nRuns^3)$ convergence rate, while the stochastic component of the regret implies $O(\sigma_H / \nRuns^{3/2} + \sigma_g / \sqrt{T})$ rate, giving us the first universal acceleration beyond first-order smoothness.

Let us conclude by discussing the intricate relationship between the universal step-size and the regret bounds. Simply put, growth of the summation in the denominator of $\gamma_t$ is of the same order as the regret bound. Under stochastic gradient and Hessian oracles, the regret bound is of order $O(T^{5/2})$, and we can trivially show using variance bounds that the step-size is lower bounded by $O(T^{-5/2})$. On the other extreme, the regret bound described in Proposition~\ref{prop:template-inequality} is bounded by a constant under deterministic oracles, which implies that the summation in the denominator of the step-size is in turn summable, i.e., the step-size has a positive, constant lower bound. This adaptive behavior of our step-size enables automatic adaptation to noise levels and thus the universal rates.

\section{Experiments} \label{sec:experiments}
In this section, we will present practical performance of \method against a set of first-order algorithms, e.g., \gd, \sgd, \adagrad~\citep{duchi2011adaptive}, \accelegrad~\citep{levy2018online}, \unixgrad~\citep{kavis2019universal}; and second-order methods, e.g., \newton, Optimal Monteiro-Svaiter (\textsc{OptMS}) \citep{carmon2022optimal}, Cubic Regularization of Newton's method (\textsc{CRN}) \citep{nesterov2006cubic} and Accelerated CRN (\textsc{ACRN}) \citep{nesterov2008accelerating} for least squares and logistic regression problems over a LIBSVM datasets, \texttt{a1a} and \texttt{a9a}. Our main objective is three-folds. First, when the objective has a favorable structure as in least squares, second-order method has cheap oracle costs 
and display superior convergence behavior. Second, we want to demonstrate the improved rates of our algorithm against accelerated and non-accelerated first-order methods through the $\ell_2$-regularized logistic regression problem. Finally, we compare our methods with respect to other second-order methods that achieve (almost) optimal rates. In the plots, the statement \textit{\# of oracle calls} on the x-axis counts any gradient or Hessian computation as one oracle call. Also note that we consider the black-box oracle model in which the algorithms only have access to gradient and Hessians without knowing the actual objective function.

When the problem is suitable, second-order methods show promising performance with truly superior run time. In Figure~\ref{subfig:leastsquares}, we display the result for least squares setting. Second-order methods are known to be suitable for quadratic problems, and our method exploits its hybrid construction to converge significantly faster than first-order methods, matching the behavior of \newton. 
\begin{figure}[ht]
\centering
     \begin{subfigure}{0.48\textwidth}
         \centering
         \includegraphics[width=\textwidth]{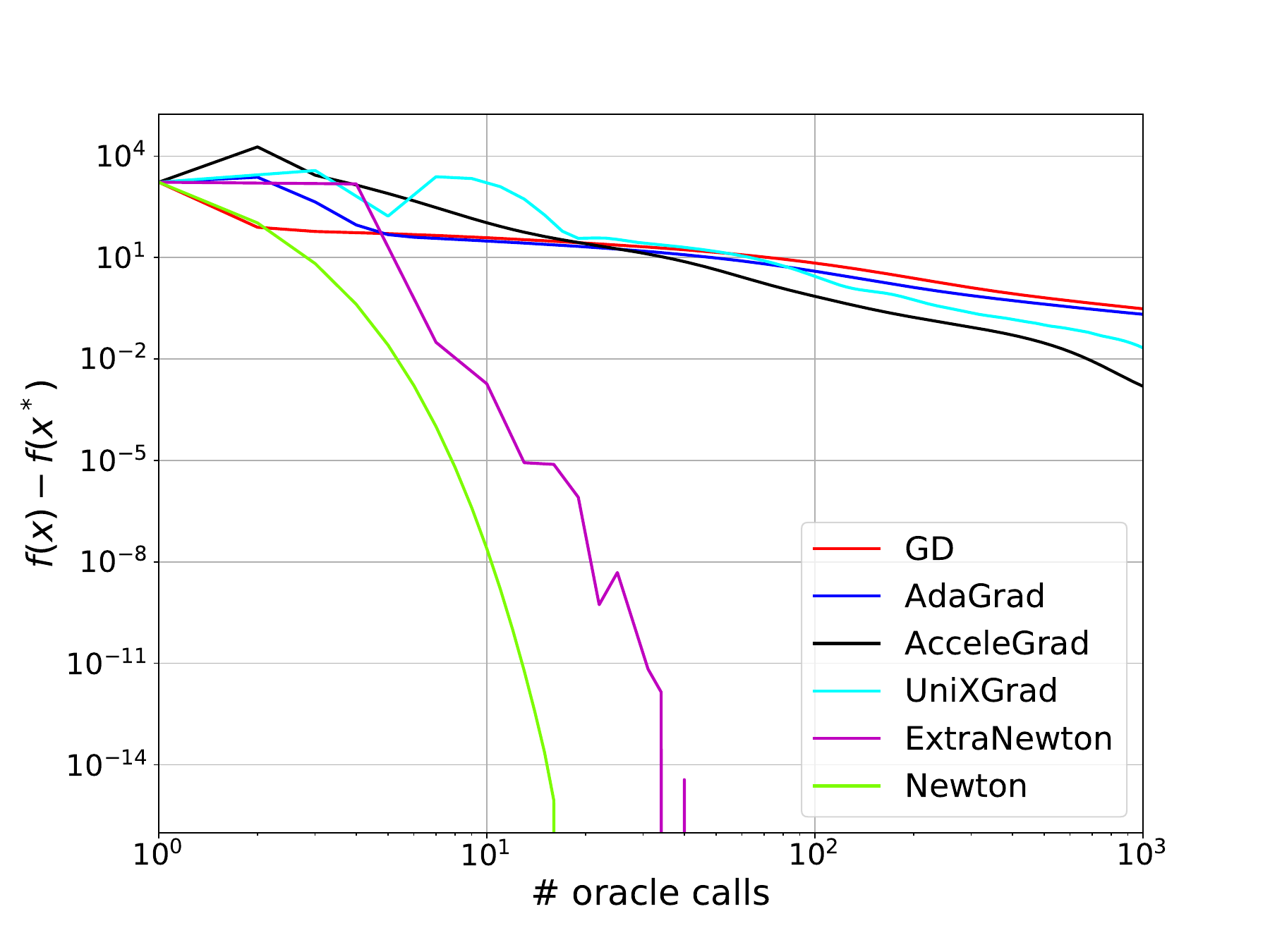}\vspace{-2mm}
         \caption{Least-squares regression on \texttt{a1a}}
         \label{subfig:leastsquares}
     \end{subfigure}
     \hfill
     \begin{subfigure}{0.48\textwidth}
         \centering
         \includegraphics[width=\textwidth]{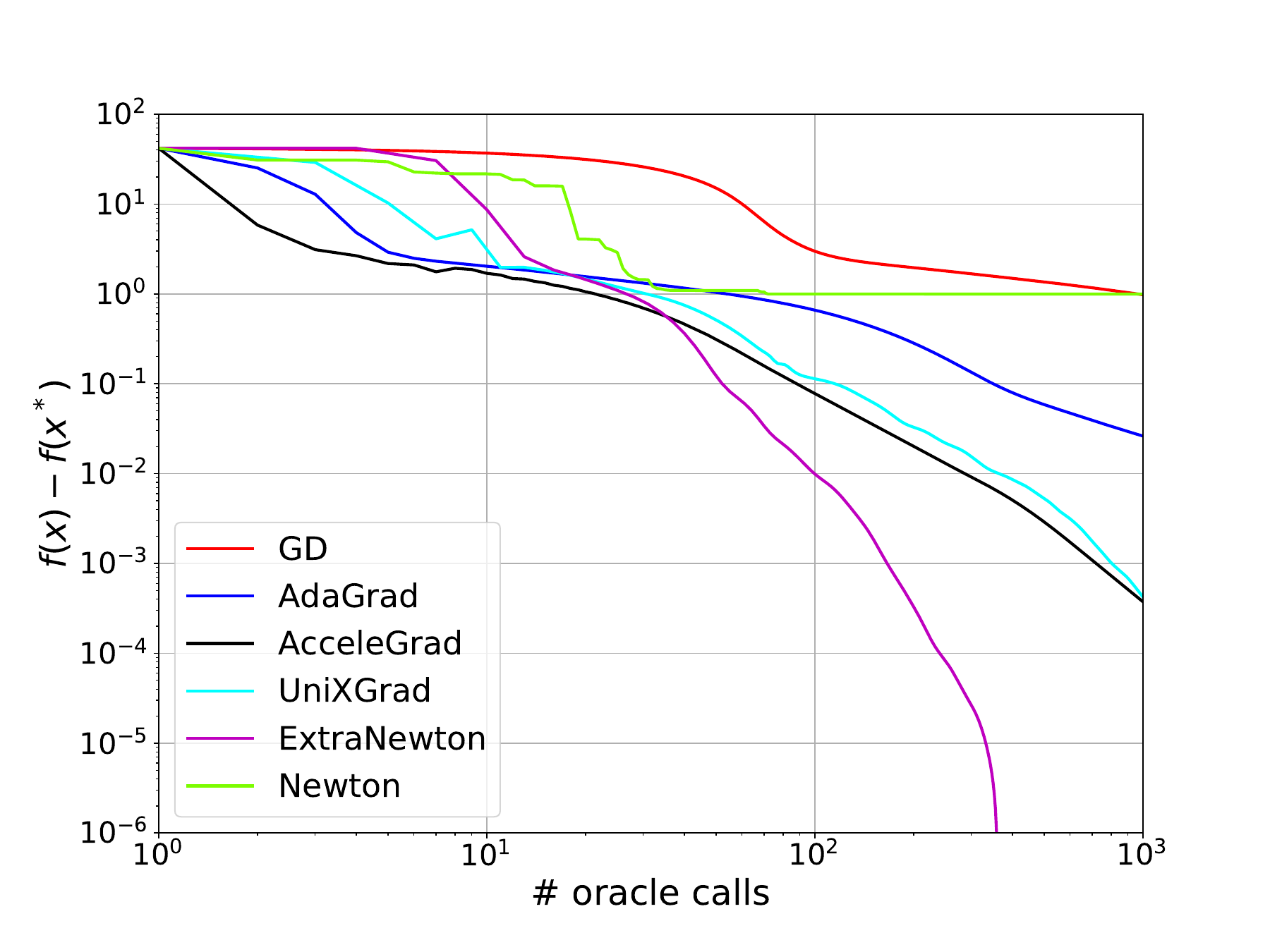}\vspace{-2mm}
         \caption{Logistic regression on \texttt{a1a}}
         \label{subfig:logreg}
     \end{subfigure}
        \caption{Comparison of value convergence for regression problems with deterministic oracle access}
        \label{fig:value-convergence}
\end{figure}

For the logistic regression problem, we regularize it with $g(x) = 1/2 \norm{x}^2$, but use a very small regularization constant to render the problem ill-conditioned, making things slightly more difficult for the algorithms \citep{marteau2019globally, mishchenko2021regularized}. Although we implement \newton with line-search, we actually observed a sporadic convergence behavior; when the initial point is close to the solution it converges similarly to \method, however when we initialize further away it doesn't converge. This non-convergent behavior has been known for \newton, even with line-search present \citep{jarre2016simple}. On the contrary, \method consistently converges; even if we perturb the initial step-size and make it adversarially large, it manages to recover due to its adaptive step-size.

\begin{figure}[ht]
\centering
     \begin{subfigure}{0.48\textwidth}
         \centering
         \includegraphics[width=\textwidth]{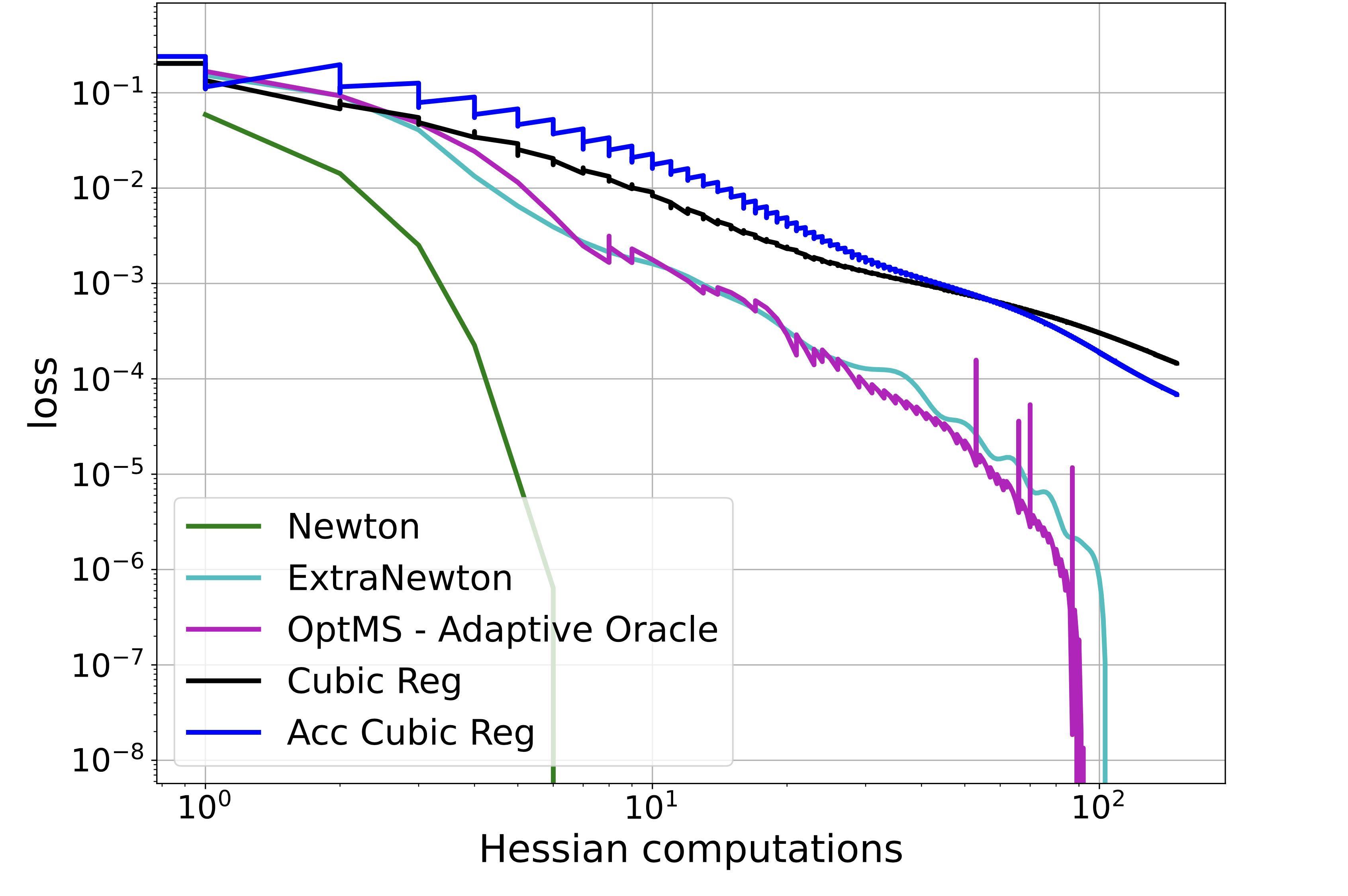}\vspace{-2mm}
         \caption{Value convergence w.r.t \# Hessian oracle calls}
         \label{subfig:hessian}
     \end{subfigure}
     \hfill
     \begin{subfigure}{0.48\textwidth}
         \centering
         \includegraphics[width=\textwidth]{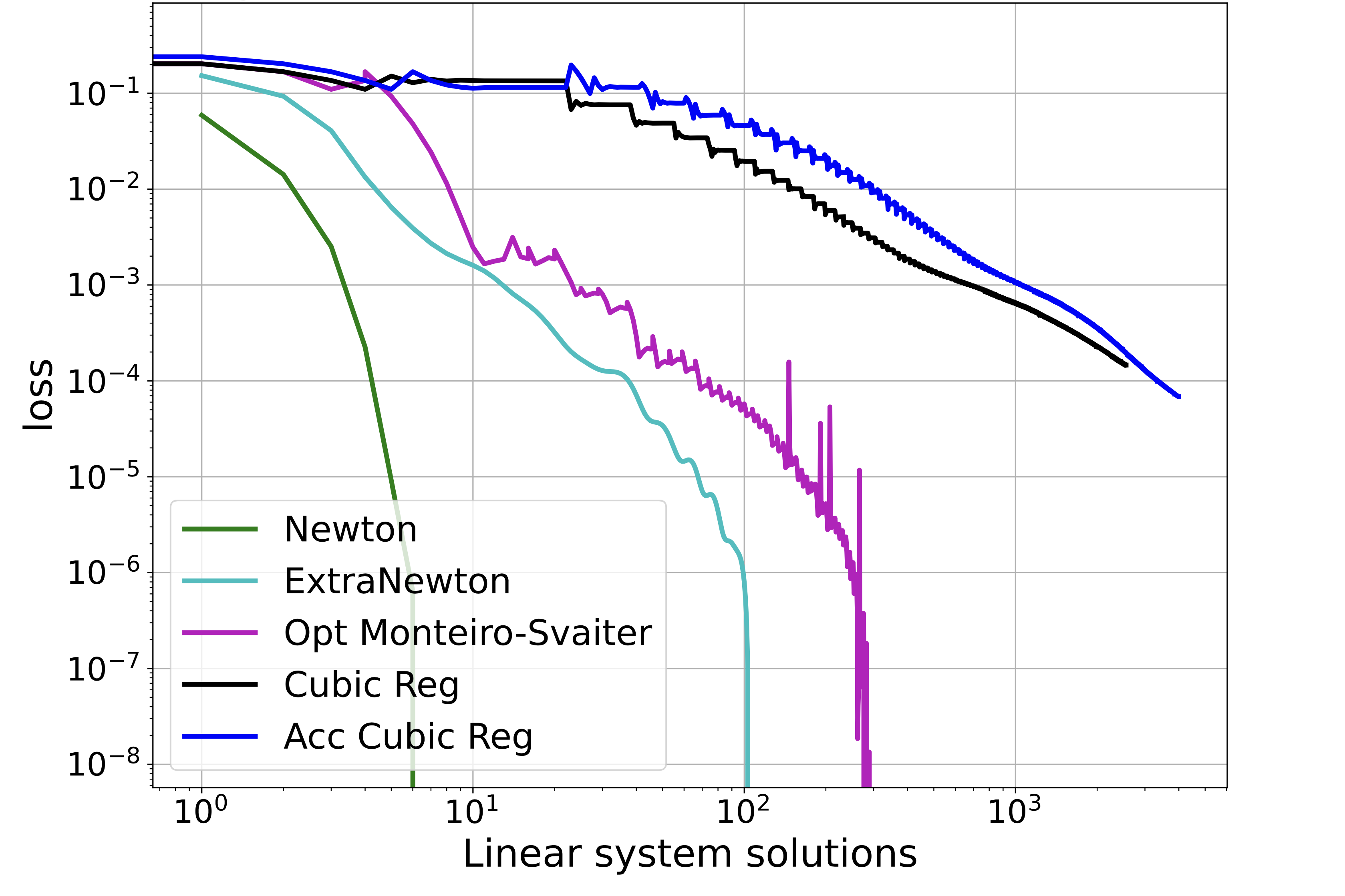}\vspace{-2mm}
         \caption{Value convergence w.r.t. \# linear system solutions}
         \label{subfig:linearsys}
     \end{subfigure}
        \caption{\method vs. second-order methods. Logistic regression with \texttt{a9a} dataset}
        \label{fig:second-order}
\end{figure}

We complement our numerical tests by comparing \method with a set of second-order methods. To that end, we implemented our method within the framework presented in \cite{carmon2022optimal}. Using the implementation and the experimental setup provided in their GitHub repository \citep{hausler2022msoptimal}, we implemented our method in their code and compared against \newton, \textsc{CRN}, \textsc{ACRN} and \textsc{OptMS} algorithms. Figure~\ref{fig:second-order} shows that \method has comparable performance to \textsc{OptMS}, which has the theoretically faster rate $O(1/\nRuns^{7/2})$, and marginally outperforms with respect to number of linear system solutions since the linesearch procedure of \textsc{OptMS} might require multiple system solutions per iteration. While \textsc{CRN} and \textsc{ACRN} has worse convergence than \method, \newton seems to have the fastest. Note that the initialization favors \newton as it lies in a close neighborhood of the solution, and \newton performance sporadically deteriorates when initialized arbitrarily.

\section{Conclusion} \label{sec:conclusion}
In this work, 
we present the \textit{first} universal, second-order algorithm, \method, which enjoys the value convergence rate of $O(\sigma_g / \sqrt{\nRuns} + \sigma_H / \nRuns^{3/2} + 1 / \nRuns^3)$. 
By extending the notion of bounded variance on stochastic gradients to stochastic \emph{Hessian}, we prove adaptation to the noise in first and second-order oracles, simultaneously, while showing accelerated rates matching that of \citet{nesterov2008accelerating} under the fully deterministic oracle model.
To that end, an important open question is whether we could design a method that achieves an improved rate interpolation guarantee $O(\sigma_g / \sqrt{\nRuns} + \sigma_H / \nRuns^{3/2} + 1 / \nRuns^{7/2})$ without depending on any line-search/bisection mechanism. 
We defer this to a future work.

\section*{Acknowledgments} \label{sec:ack}
This project has received funding from the European Research Council (ERC) under the European Union's Horizon 2020 research and innovation programme (grant agreement n° $725594$ - time-data) and the Swiss National Science Foundation (SNSF) under grant number $200021\_205011$.

\bibliography{refs, bibtex/Bibliography-PM}
\bibliographystyle{plainnat}




\newpage
\appendix

\par\noindent\rule{\textwidth}{1pt}
\begin{center}
	\vspace{-2mm}
    \LARGE APPENDIX
    \vspace{-2mm}
\end{center}
\par\noindent\rule{\textwidth}{1pt}

%

\section{Preface} \label{app:prelude}

In \appref{app:notation}, we provide a complete list of notation and definitions that we have used throughout the manuscript.

In \appref{app:experiments}, we showcase additional numerical evidence for the comparison we provided in the main text. Due to space constraints, we moved most of our plots to the appendix. We investigate the practical behavior in both deterministic and stochastic setting.

In \appref{app:conversion}, we begin with the proof of the generalized online-to-batch conversion in \thmref{thm:conversion} to form the connection between the offline regret $\Reg(\sol[\point])$ and value convergence $\obj( \lastlead[\stateavg] ) - \obj(\sol[\point])$.

Then in \appref{app:template-inequality}, we present the analysis for obtaining the template regret bound in Proposition~\ref{prop:template-inequality}. This template inequality is indeed the point where the analysis in the deterministic, implicit setting and universal, explicit setting part ways.

In \appref{app:numerical}, we take a small detour to introduce a crucial numerical inequality that is commonly used in the analysis of adaptive methods.

We present the universal convergence analysis of \method (\thmref{thm:explicit}) in \appref{app:explicit}.

In \appref{app:implicit}, we share the analysis of our conceptual framework: convergence of the implicit algorithm \eqref{eq:implicit} in deterministic setting (\thmref{thm:implicit}), with the appropriate corollary of Proposition~\ref{prop:template-inequality} for the case of deterministic oracles in this section.

\section{Notation and Definitions} \label{app:notation}

To complement the notation in \secref{sec:problem-setup}, we will present a complete list of definitions and parameter descriptions to make it easier for the reader to follow the technical arguments in the whole paper.
\renewcommand{\arraystretch}{1.5}
\begin{table}[H]
\centering
\caption{A complete list of parameters and expressions, their definitions and descriptions}
\vspace{1mm}
\begin{tabular}{l|l|l} \label{tbl:app_notation}
	&\multicolumn{1}{|c|}{\textbf{Formal Definition}}
	&\multicolumn{1}{|c}{\textbf{Description}}
	\\
\hline
$\obj$ & $\obj : \mathbb R^d \to \mathbb R + \bc{ +\infty } $ & objective function\\
\hline
$\compact$ & $\compact \subset \mathbb R^d$ & convex and compact constraint set\\
\hline
$\sol$ & $= \argmin_{\point \in \compact} \obj(\point)$ & solution of the constrained problem \eqref{eq:opt}\\
\hline
$\eucdiam$ & $= \sup_{x, y \in \compact} \norm{x - y} $ & diameter of the constraint set $\compact$\\
\hline
$\smooth$ & $\|\nabla^{2}f(\point)-\nabla^{2}f(\pointalt)\| \leq \smooth \|\point-\pointalt\|$ & second-order smoothness constant of $\obj$\\
\hline
$\sgrad(\cdot, \xi)$ & $\Expect{ \sgrad(\point, \xi) \mid \point } = \nabla \obj(\point)$, $\,\,x \indep \xi$ & unbiased gradient estimate\\
\hline
$\shess(\cdot, \xi)$ & $\Expect{ \shess(\point, \xi) \mid \point } = \nabla^2 \obj(\point)$, $\,\,x \indep \xi$ & unbiased Hessian estimate\\
\hline
$\curr[\mathcal F]$ & $= \sigma(\xi_1, \xi_{1 + \frac{1}{2}}, \cdots, \curr[\xi])$ & $\sigma$-algebra generated by random variables up to $\curr[\xi]$\\
\hline
$\lead[\mathcal F]$ & $= \sigma(\xi_1, \xi_{1 + \frac{1}{2}}, \cdots, \curr[\xi], \lead[\xi])$ & $\sigma$-algebra generated by random variables up to $\lead[\xi]$\\
\hline
$\vargrad$ & $\Expect{ \norm{ \sgrad(\point) - \nabla \obj (\point) }^2 \mid \point } \leq \vargrad^2$ & variance bound for gradient estimate\\
\hline
$\varhess$ & $\Expect{ \norm{ \shess(\point) - \nabla^2 \obj (\point) }^2 \mid \point } \leq \varhess^2$ & variance bound for Hessian estimate\\
\hline
$\sigma$ & $= \max \bc{ \vargrad, \varhess }$ & maximum variance of oracles\\
\hline
$\curr[\stepalt]$ & Eq.~\eqref{eq:adaptive-stepsize-implicit} and Eq.~\eqref{eq:adaptive-stepsize-explicit} & adaptive step-size\\
\hline
$\curr[\wgrad]$ & $= \run^2$ & gradient weights\\
\hline
$\curr[\wgradsum]$ & $= \sum_{\runalt=\start}^\run \iter[\wgrad]$ & normalization factor for gradient weights $\curr[\wgrad]$\\
\hline
$\curr[\wavg]$ & $= \run^p$, where $p \geq 2$ & averaging weights\\
\hline
$\curr[\wavgsum]$ & $= \sum_{\runalt=\start}^\run \iter[\wavg]$ & normalization factor for averaging weights $\curr[\wavg]$\\
\hline
\end{tabular}
\end{table}
\renewcommand{\arraystretch}{1}

\section{Further Experimental Evaluation} \label{app:experiments}
In this section we will present additional numerical experiments in two fronts;
\begin{itemize}
    \item[-] we run logistic regression and least-squares regression under deterministic gradients with another LIBSVM datasets, \texttt{w1a}, 
    \item[-] and subsequently display results in the stochastic setting for the same datasets \texttt{a1a, w1a}. 
\end{itemize}
Figure~\ref{fig:app-determinisitic} shows the results for the deterministic experiments while Figure~\ref{fig:app-stochastic} focuses on the results of the stochastic setting. In both figures, we present results for the least-squares in the first column and the logistic regression in the second column.

In Figure~\ref{fig:app-determinisitic}, the x-axis represent the number of calls made to the \textit{deterministic} oracle, and in Figure~\ref{fig:app-stochastic}, x-axis corresponds to \textit{number of full data passes (epochs)} to compute the stochastic gradient estimates. The deterministic setup is the same as we described in the main text. In the case of stochastic gradients, we compute mini-batch gradient estimates with a batch-size of 50 samples. We plot the mean of 5 trials for all the methods under mini-batch gradients and also display the variance as the shaded region around the mean curve.

For the case of logistic regression under deterministic gradients, our method performs better than the rest of the pack with \texttt{a1a} dataset but has almost matching performance with a smaller performance gap compared to accelerated first-order methods with \texttt{w1a} dataset. For both datasets, we tried to tune Newton's method for a randomly-chosen initialization but it was very difficult to find a parameter setting where Newton shows any reasonable behavior. One could notice that Newton's method doesn't converge to the solution for logistic regression problem for this random initial point.

\begin{figure}[H]
\centering
     \begin{subfigure}{0.48\textwidth}
         \centering
         \includegraphics[width=\textwidth]{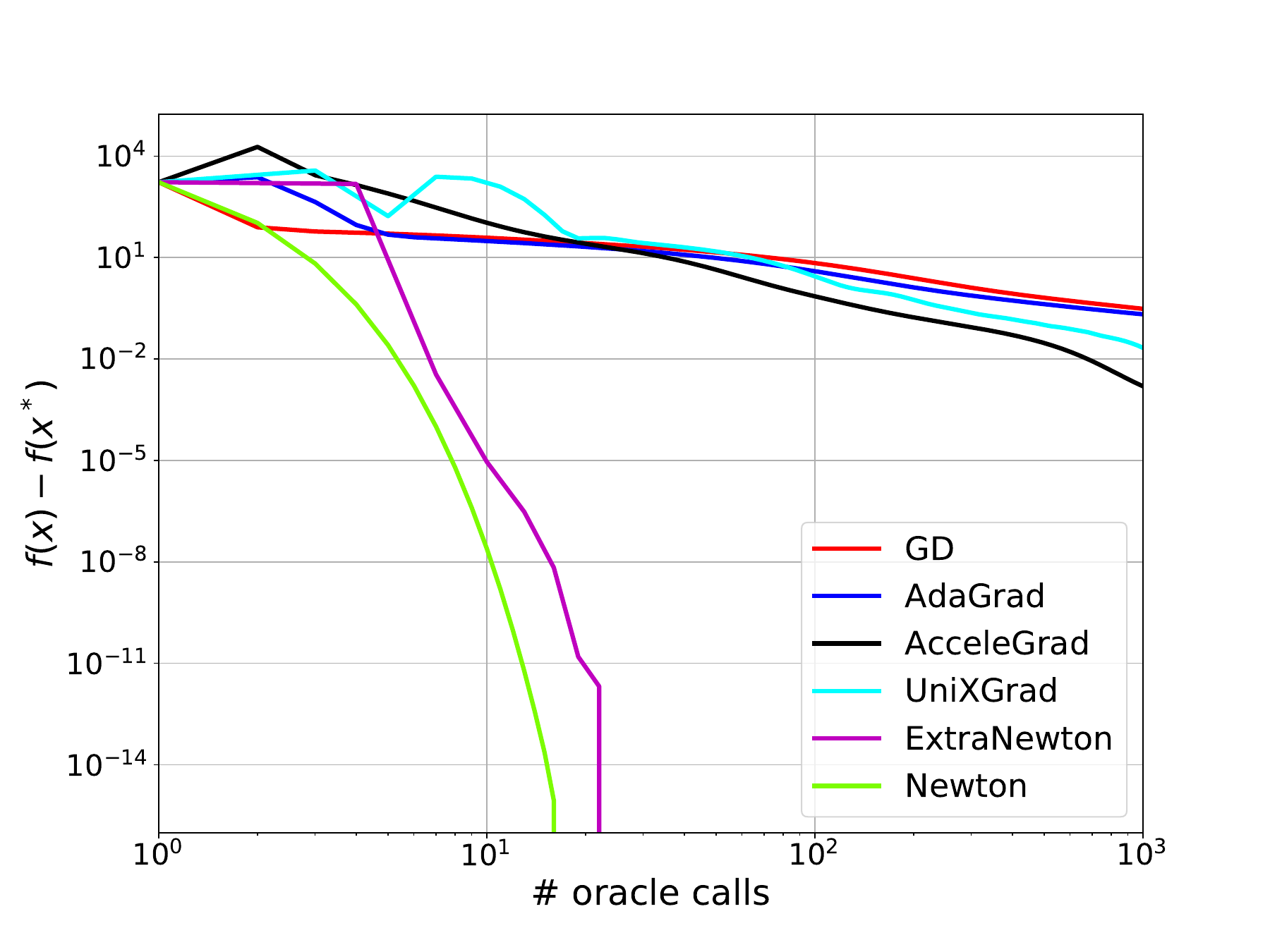}\vspace{-2mm}
         \caption{Least-squares, \texttt{\textbf{a1a}}}
         \label{subfig:app-leastsq-a1a}
     \end{subfigure}
     \begin{subfigure}{0.48\textwidth}
         \centering
         \includegraphics[width=\textwidth]{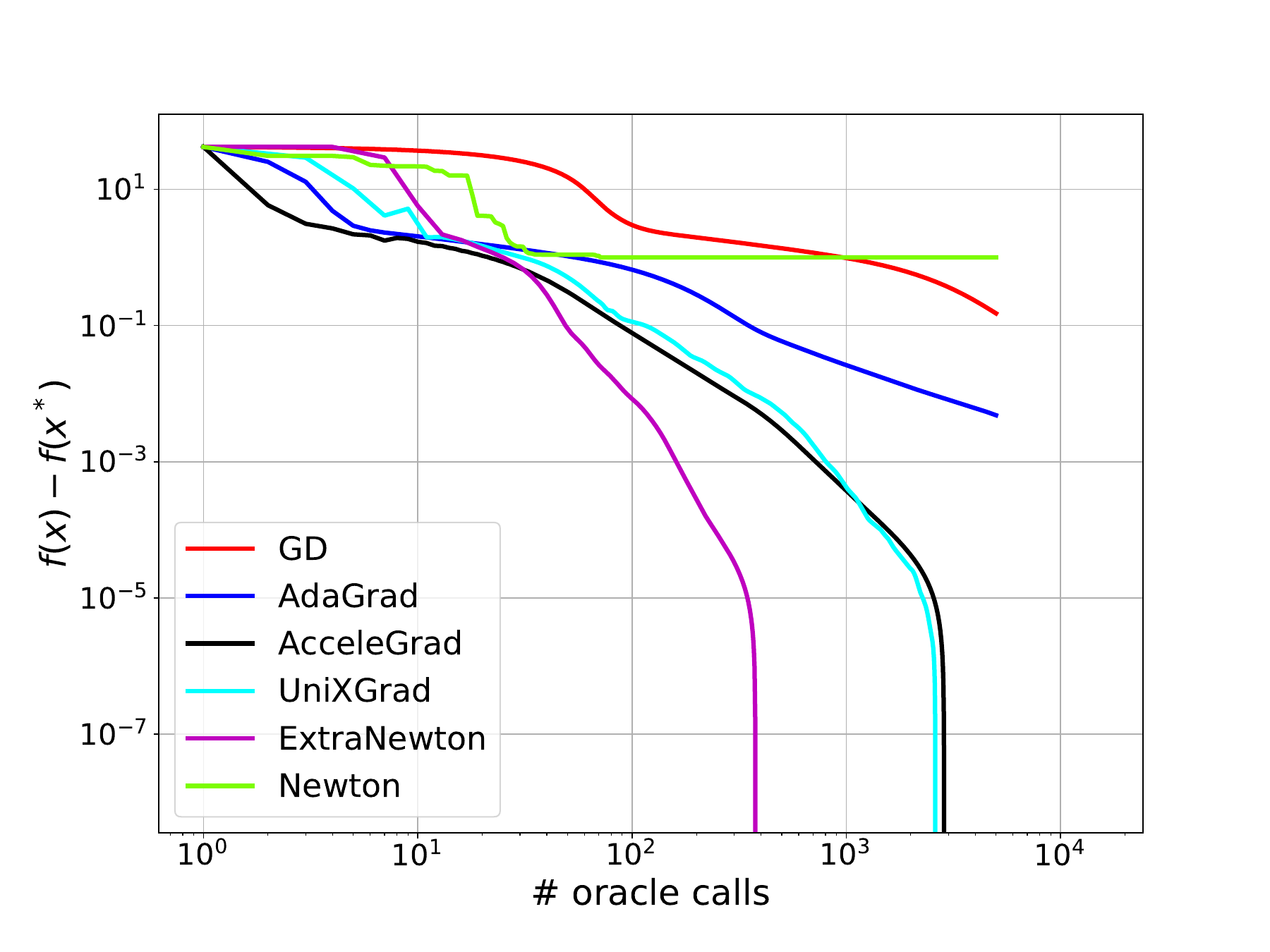}\vspace{-2mm}
         \caption{Logistic regression, \texttt{\textbf{a1a}}}
         \label{subfig:app-logreg-a1a}
     \end{subfigure}
     \begin{subfigure}{0.48\textwidth}
         \centering
         \includegraphics[width=\textwidth]{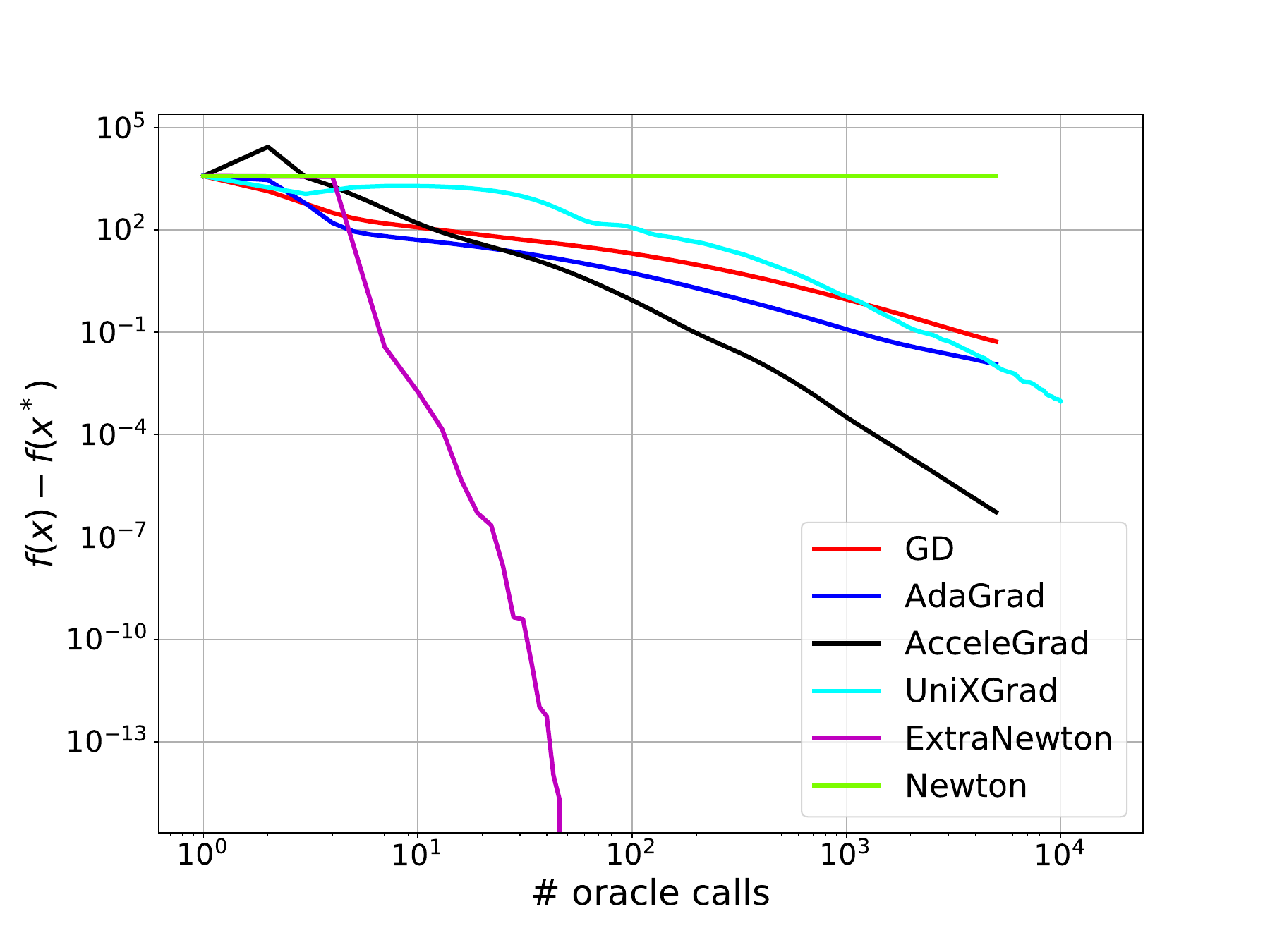}\vspace{-2mm}
         \caption{Least-squares, \texttt{\textbf{w1a}}}
         \label{subfig:app-leastsq-w1a}
     \end{subfigure}
     \begin{subfigure}{0.48\textwidth}
         \centering
         \includegraphics[width=\textwidth]{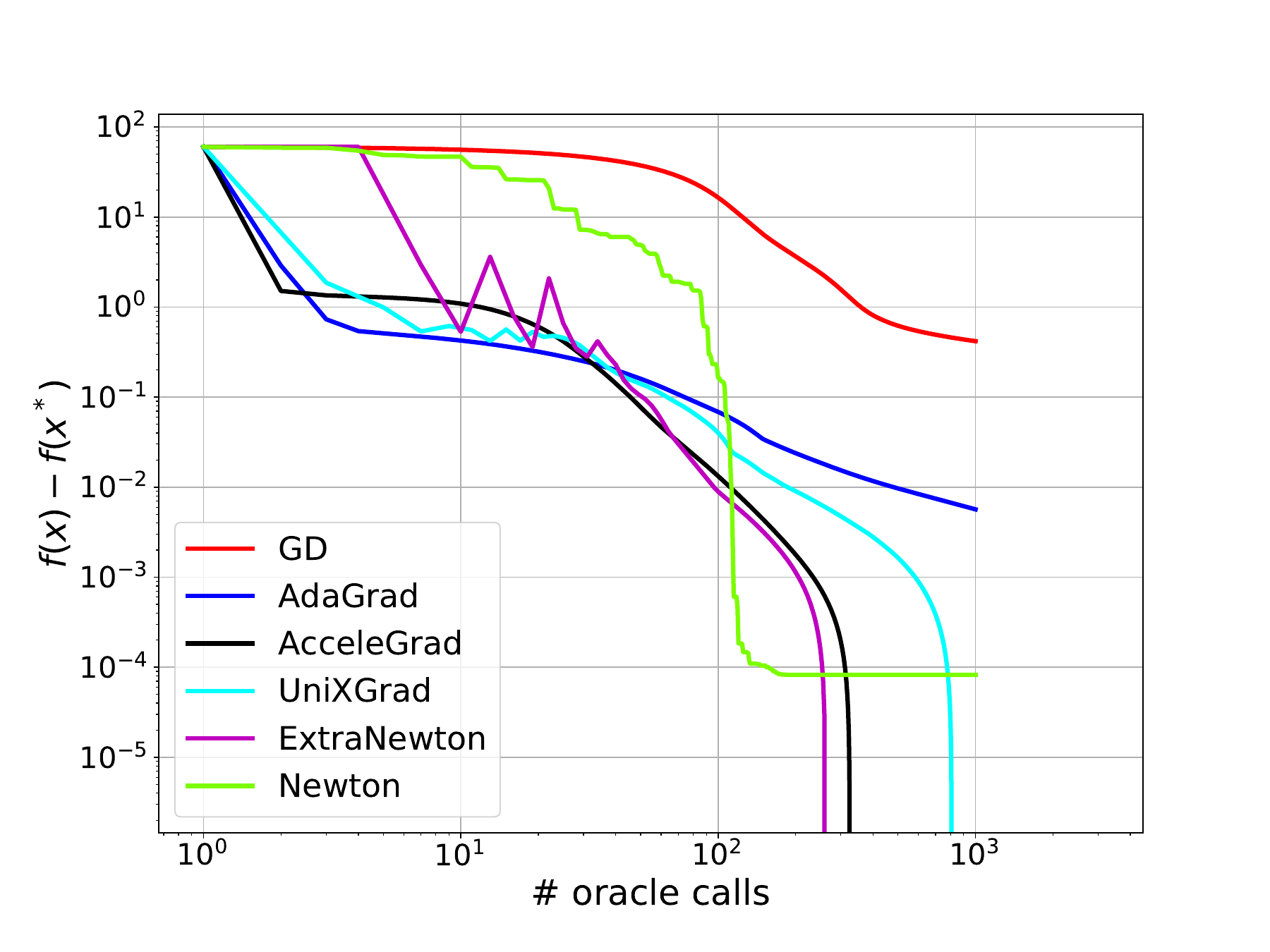}\vspace{-2mm}
         \caption{Logistic regression, \texttt{\textbf{w1a}}}
         \label{subfig:app-logreg-w1a}
     \end{subfigure}
        \caption{Comparison of value convergence for regression problems with \textbf{deterministic} oracle access}
        \label{fig:app-determinisitic}
\end{figure}

We observe that the main advantage of our approach, and in general that of second-order methods, becomes apparent when the problem at hand has a compatible structure such as least-squares. Intuitively, second-order methods should benefit when the cost of computing the Hessian is comparable to gradient computation. In fact, quadratic problems like least-squares yield a constant Hessian for any point in the domain, granting a significant advantage to second-order methods. We exemplify this behavior for least-squares problem with deterministic oracles. With \texttt{w1a} dataset, we couldn't get Newton's method to converge once again. On the contrary, our method shows significant performance upgrade compared to first-order methods while converging consistently in all our trials.

Finally, we have the experiments under stochastic oracles. We essentially present these results for two main reasons; to show that our method works seamlessly with stochastic gradients without any modifications, and to demonstrate that \method achieves the $O(1/\sqrt{T})$ rate (same as other methods we compare against) when the gradient information is noisy. We showcase both of these perspectives in Figure~\ref{fig:app-stochastic}.

\begin{figure}[H]
\centering
     \begin{subfigure}{0.48\textwidth}
         \centering
         \includegraphics[width=\textwidth]{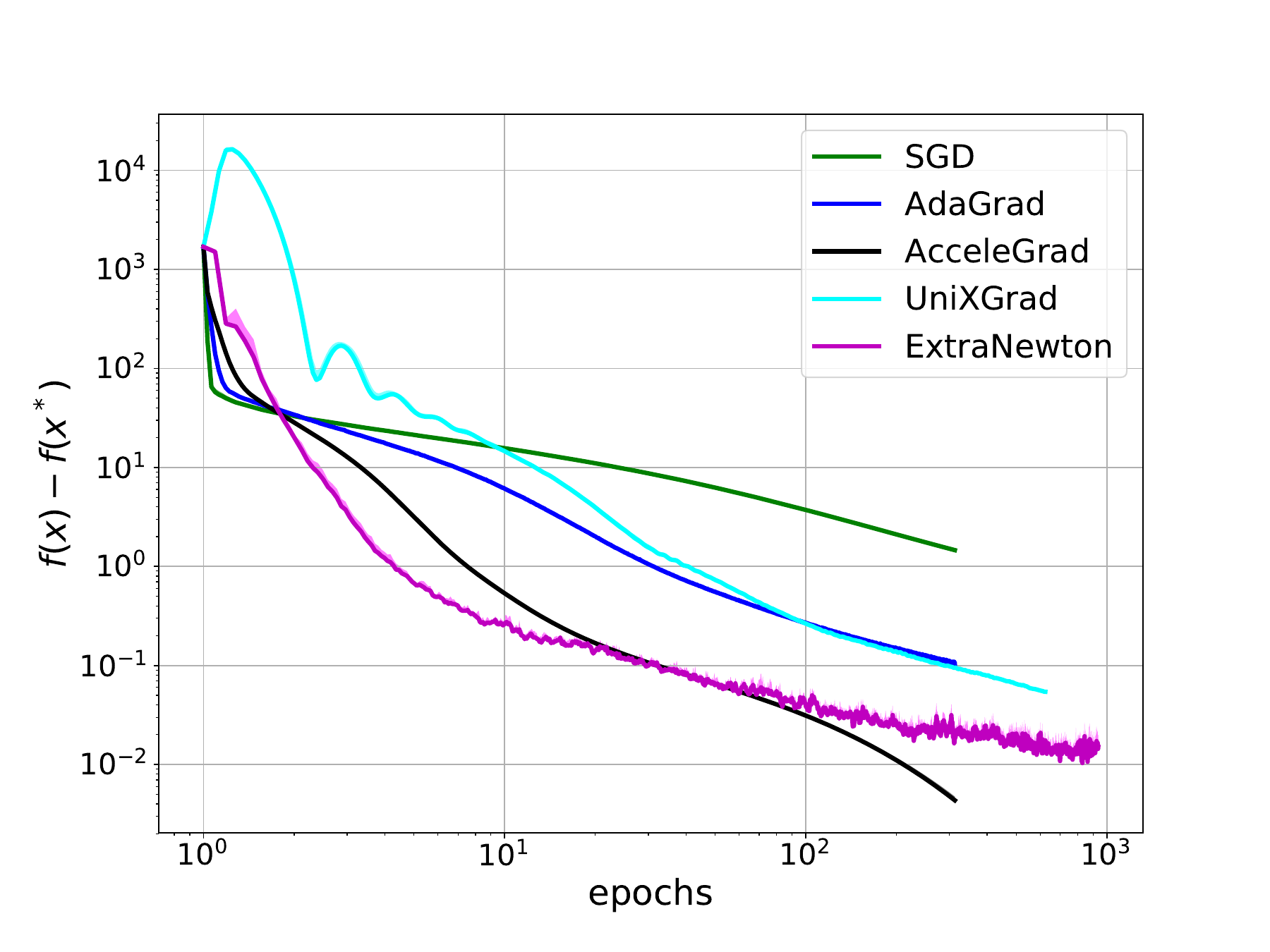}\vspace{-2mm}
         \caption{Least-squares, \texttt{\textbf{a1a}}}
         \label{subfig:app-leastsq-a1a-stoch}
     \end{subfigure}
     \begin{subfigure}{0.48\textwidth}
         \centering
         \includegraphics[width=\textwidth]{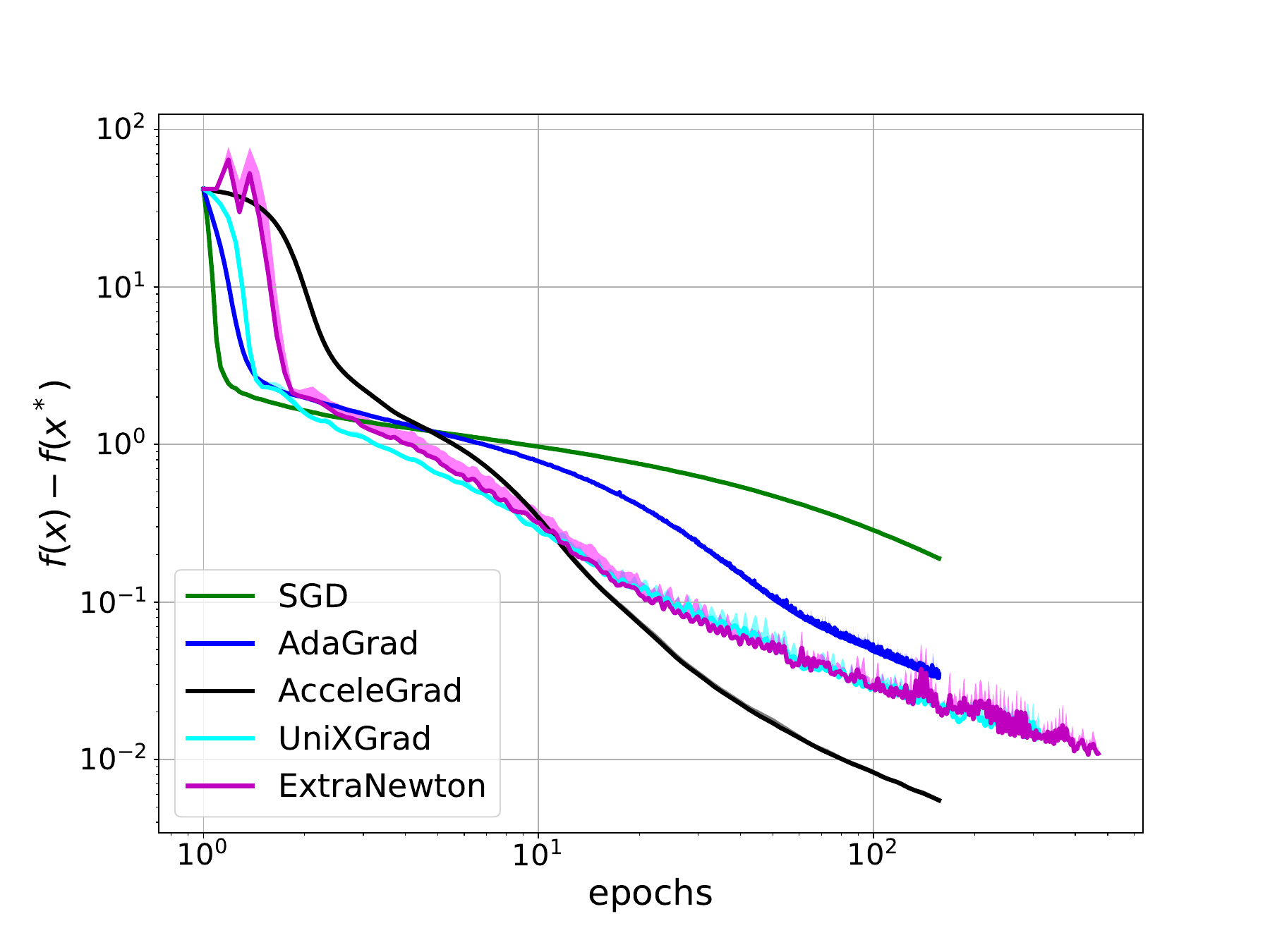}\vspace{-2mm}
         \caption{Logistic regression, \texttt{\textbf{a1a}}}
         \label{subfig:app-logreg-a1a-stoch}
     \end{subfigure}
     \begin{subfigure}{0.48\textwidth}
         \centering
         \includegraphics[width=\textwidth]{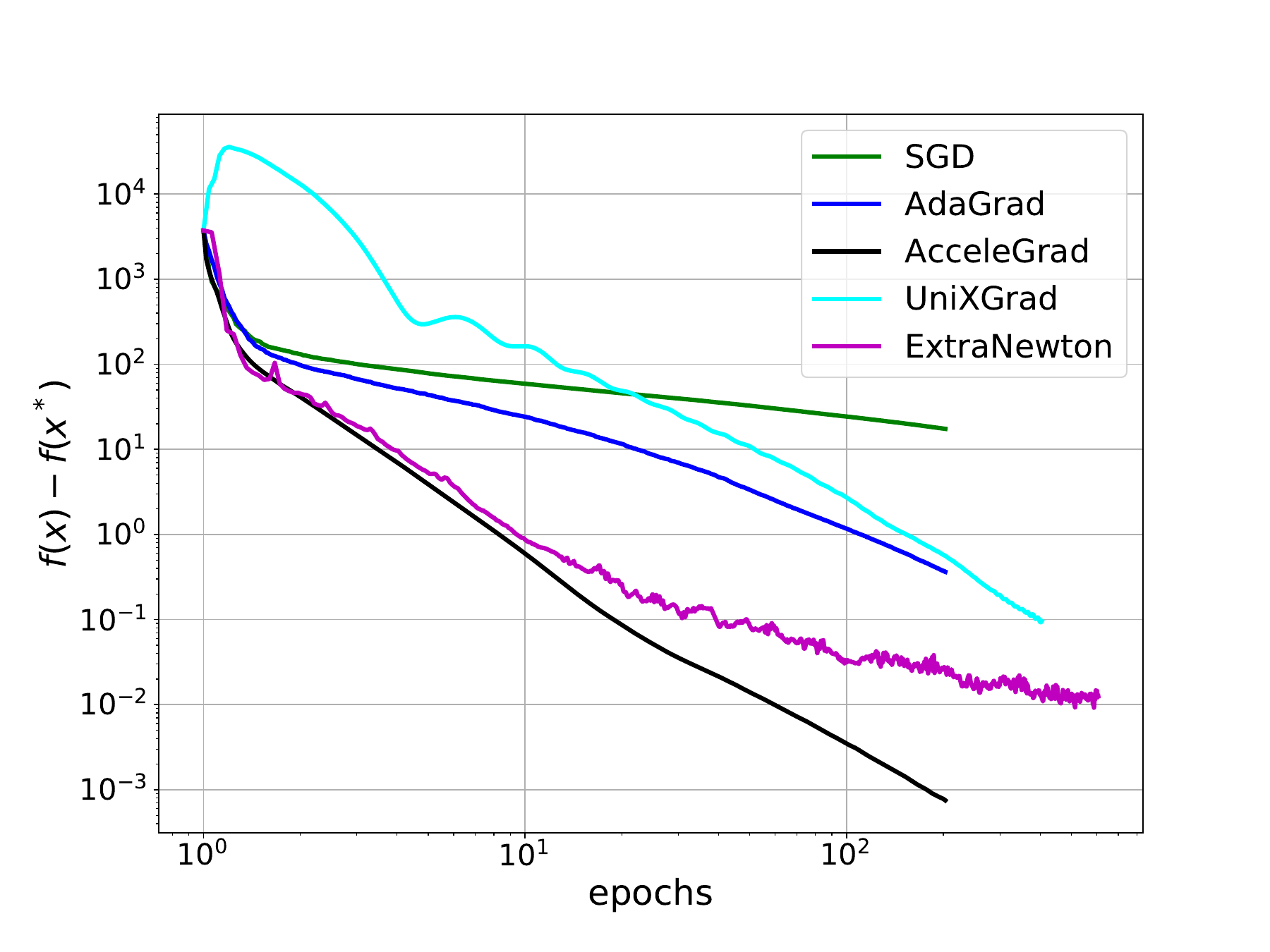}\vspace{-2mm}
         \caption{Least-squares, \texttt{\textbf{w1a}}}
         \label{subfig:app-leastsq-w1a-stoch}
     \end{subfigure}
     \begin{subfigure}{0.48\textwidth}
         \centering
         \includegraphics[width=\textwidth]{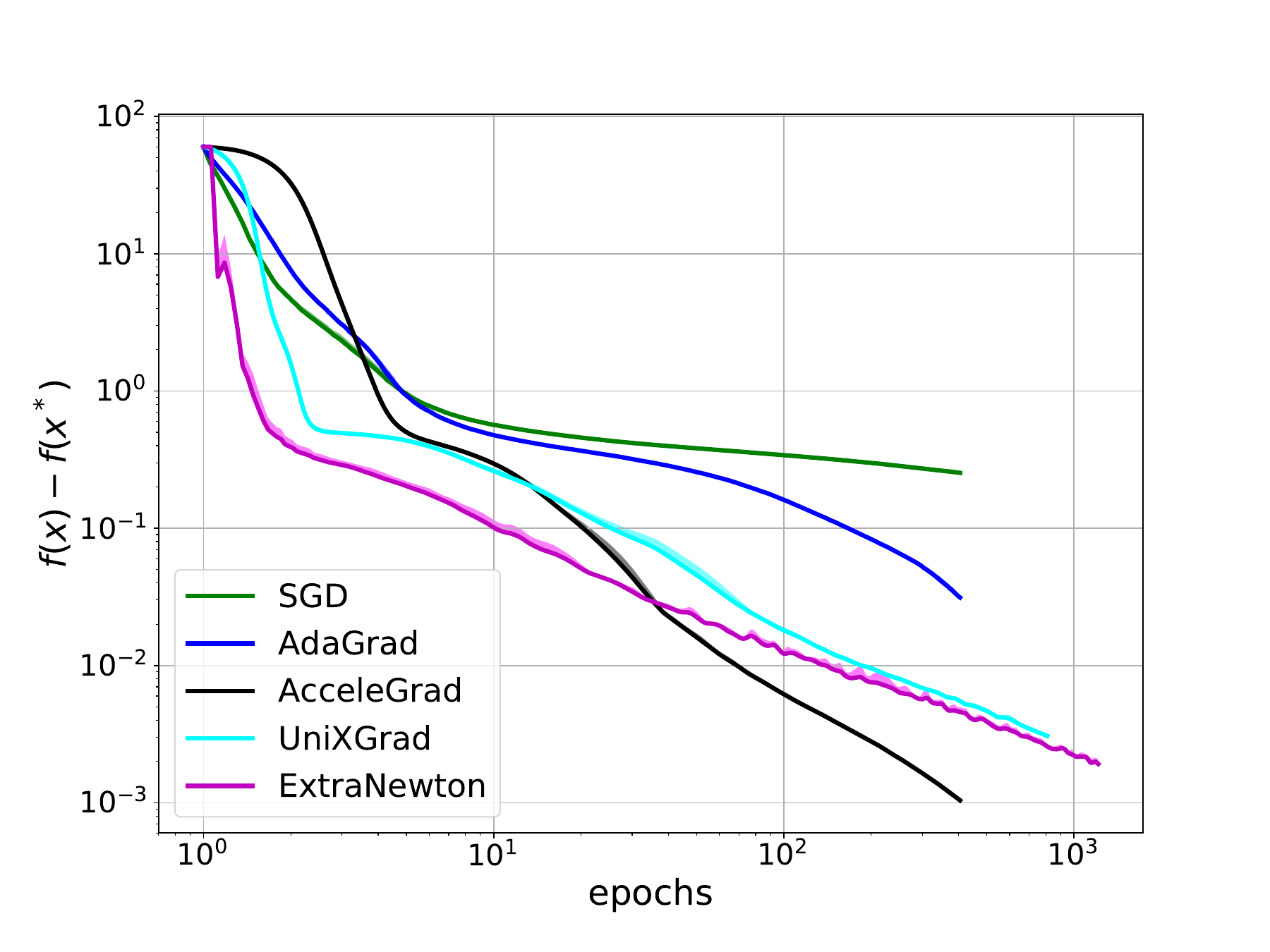}\vspace{-2mm}
         \caption{Logistic regression, \texttt{\textbf{w1a}}}
         \label{subfig:app-logreg-w1a-stoch}
     \end{subfigure}
        \caption{Comparison of value convergence for regression problems with \textbf{stochastic} oracle access}
        \label{fig:app-stochastic}
\end{figure}

\section{Generalized Online-to-batch Conversion (\thmref{thm:conversion})} \label{app:conversion}

In this section we present the online-to-batch conversion scheme which connects 
the optimality gap $\obj(\lastlead[\stateavg])-\obj(\sol)$ with the "weighted" regret $\Reg(\sol) = \sum_{\run=1}^{\nRuns}\curr[\wgrad] \ip{ \nabla \obj (\lead[\stateavg]) }{ \lead - \sol } $.

\begin{reptheorem}{thm:conversion}
    Let $\Reg(\sol[\point])$ denote the anytime regret for the decision sequence $\{ \lead \}_{\run = \start}^{\nRuns}$ as in \eqref{eq:regret}, and define two sequences of non-decreasing weights $\curr[\wgrad]$ and  $\curr[\wavg]$ such that $\curr[\wgrad], \curr[\wavg] \geq 1$. As long as $ \curr[\wgrad] / \curr[\wavg]$ is ensured to be non-increasing, 
    \begin{align*}
        f(\lastlead[\stateavg]) - f(\sol[\point]) \leq \frac{\Reg(\sol[\point])}{ \last[\wgrad] \frac{\last[\wavgsum]}{\last[\wavg]}}
    \end{align*}
\end{reptheorem}
\begin{proof}
First, recall the definition of the offline regret:
\begin{align*}
    \Reg(\sol[\point]) = \sum_{\run=\start}^{\nRuns} \curr[\wgrad] \ip{\nabla \obj(\lead[\stateavg]}{\lead - \sol[\point]}
\end{align*}
Devising our analysis in the spirit of \citep{cutkosky2019anytime, kavis2019universal}, we need to relate $\lead$ to the average iterate $\lead[\stateavg]$ in order to exploit the convexity of the objective function. Notice that we could write the iterate $\lead$ as the difference of consecutive \textit{average} iterates,
\begin{align} \label{def:A-t}
	\curr[\wgrad] \lead = \curr[\wgrad] \frac{\curr[\wavgsum]}{\curr[\wavg]} \lead[\stateavg] - \curr[\wgrad] \frac{\prev[\wavgsum]}{\curr[\wavg]} \beforelead[\stateavg].
\end{align}
Also, we could subsequently express $\curr[\wgrad] \sol[\point] = \curr[\wgrad] \frac{\curr[\wavgsum]}{\curr[\wavg]} \sol[\point] - \curr[\wgrad] \frac{\prev[\wavgsum]}{\curr[\wavg]} \sol[\point]$. Combining them together,
\begin{align*}
	\Reg(\sol[\point]) &= \sum_{\run=\start}^{\nRuns} \curr[\wgrad] \ip{\nabla \obj(\lead[\stateavg]}{\lead - \sol[\point]}\\
	&= \sum_{\run=\start}^{\nRuns} \curr[\wgrad] \frac{\curr[\wavgsum]}{\curr[\wavg]} \ip{\nabla \obj(\lead[\stateavg])}{ \lead[\stateavg] - \sol[\point] } - \curr[\wgrad] \frac{\prev[\wavgsum]}{\curr[\wavg]} \ip{\nabla \obj (\lead[\stateavg]}{ \beforelead[\stateavg] - \sol[\point] } \\
	&= \sum_{\run=\start}^{\nRuns} \curr[\wgrad] \ip{\nabla \obj (\lead[\stateavg])}{ \lead[\stateavg] - \sol[\point] } + \curr[\wgrad] \frac{\prev[\wavgsum]}{\curr[\wavg]} \ip{\nabla \obj(\lead[\stateavg])}{ \lead[\stateavg] - \beforelead[\stateavg] } \\
\end{align*}
where we added and subtracted $\curr[\wgrad] \frac{\prev[\wavgsum]}{\curr[\wavg]} \ip{\nabla \obj(\lead[\stateavg])}{ \lead[\stateavg] }$ to obtain the second equality. Having expressed both inner products in the form we want, we could apply convexity and telescope.
\begin{align*}
	\sum_{\run=\start}^{\nRuns} \curr[\wgrad] &\ip{\nabla \obj(\lead[\stateavg]}{\lead - \sol[\point]}\\
	&\geq \sum_{\run=\start}^{\nRuns} \curr[\wgrad] \br{ \obj ( \lead[\stateavg] ) - \obj ( \sol[\point] ) } + \curr[\wgrad] \frac{\prev[\wavgsum]}{\curr[\wavg]} \br{ \obj ( \lead[\stateavg] ) - \obj ( \beforelead[\stateavg] ) } \\
	&= \sum_{\run=\start}^{\nRuns} \curr[\wgrad] \br{ \obj ( \lead[\stateavg] ) - \obj ( \sol[\point] ) } + \curr[\wgrad] \frac{\prev[\wavgsum]}{\curr[\wavg]} \br{ \obj ( \lead[\stateavg] ) - \obj ( \sol[\point] ) } - \curr[\wgrad] \frac{\prev[\wavgsum]}{\curr[\wavg]} \br{ \obj ( \beforelead[\stateavg] ) - \obj ( \sol[\point] ) } \\
	&= \sum_{\run=\start}^{\nRuns} \curr[\wgrad] \frac{\curr[\wavgsum]}{\curr[\wavg]} \br{ \obj ( \lead[\stateavg] ) - \obj ( \sol[\point] ) } - \curr[\wgrad] \frac{\prev[\wavgsum]}{\curr[\wavg]} \br{ \obj ( \beforelead[\stateavg] ) - \obj ( \sol[\point] ) } \\
	&= \last[\wgrad] \frac{\last[\wavgsum]}{\last[\wavg]} \br{ \obj ( \lastlead[\stateavg] ) - \obj ( \sol[\point] ) } - \init[\wgrad] \frac{\beforeinit[\wavgsum]}{\init[\wavg]} \br{ \obj( \stateavg_{-1/2} ) - \obj( \sol[\point] ) } + \sum_{\run=\start}^{\nRuns-1} \curr[\wavgsum] \br{ \frac{\curr[\wgrad]}{\curr[\wavg]} - \frac{\next[\wgrad]}{\next[\wavg]} }  \br{ \obj ( \lead[\stateavg] ) - \obj ( \sol[\point] ) }
\end{align*}
Setting $\beforeinit[\wavgsum] = 0$ eliminates the second term. To conclude the proof, we need to show that the summation term in the above expression is always non-negative. This is ensured when the sequence $\frac{\curr[\wgrad]}{\curr[\wavg]}$ is monotonically non-increasing, which is specified in the theorem statement (and subsequently satisfied by the algorithms). 
Hence,
\begin{align*}
	\sum_{\run=\start}^{\nRuns} \curr[\wgrad] &\ip{\nabla \obj(\lead[\stateavg]}{\lead - \sol[\point]}\\
	&= \last[\wgrad] \frac{\last[\wavgsum]}{\last[\wavg]} \br{ \obj ( \lastlead[\stateavg] ) - \obj ( \sol[\point] ) } + \sum_{\run=\start}^{\nRuns-1} \curr[\wavgsum] \br{ \frac{\curr[\wgrad]}{\curr[\wavg]} - \frac{\next[\wgrad]}{\next[\wavg]} }  \br{ \obj ( \lead[\stateavg] ) - \obj ( \sol[\point] ) }\\
	&\geq \last[\wgrad] \frac{\last[\wavgsum]}{\last[\wavg]} \br{ \obj ( \lastlead[\stateavg] ) - \obj ( \sol[\point] ) }.
\end{align*}
Rearranging the terms gives us the final result
\begin{align*}
	\obj ( \lastlead[\stateavg] ) - \obj ( \sol[\point] ) \leq \frac{ \sum_{\run=\start}^{\nRuns} \curr[\wgrad] \ip{\nabla \obj(\lead[\stateavg])}{\lead - \sol[\point]} }{ \last[\wgrad]\frac{\last[\wavgsum]}{\last[\wavg]} } = \frac{ \Reg(\sol[\point]) }{ \last[\wgrad]\frac{\last[\wavgsum]}{\last[\wavg]} }.
\end{align*}

\end{proof}

\section{Template Regret Bound (Proposition~\ref{prop:template-inequality})} \label{app:template-inequality}

In this section, we will prove the template inequality in Proposition~\ref{prop:template-inequality} in the case of stochastic oracles. This inequality will give us the main departure point for both \thmref{thm:implicit} and \thmref{thm:explicit}. We will prove a corollary of the following result later on, specifically for the deterministic setup, which will follow the same steps as Proposition~\ref{prop:template-inequality}.

For ease of navigation, we present \method once more.
\begin{algorithm}[H]
	\caption*{\method} \label{alg:app-explicit}
	\textbf{Input}: $\init[\state] \in \compact$, $\,\curr[\wgrad] = t^2$ and $A_t= \sum_{s=1}^{t} \iter[\wgrad]$, $\,\curr[\wavg] = t^p$ ($p\geq2$) and $B_t= \sum_{s=1}^{t} \iter[\wavg]$, $\,\gamma > 0$, $\xi_t \sim$ i.i.d.\\[-3mm]
	\begin{algorithmic}[1]
    	\FOR {$t=1$ to $T$}
                \smallskip
                \STATE $~~~~~~\curr[\stepalt] = \dfrac{\step}{\sqrt{ \stepscale + \sum_{\runalt=\start}^{\run-1} \iter[\wgrad]^2 \norm{\sgrad(\iterlead[\stateavg], \xi_{s + \frac{1}{2}}) - \Fs(\iterlead[\stateavg]; \iter[\stateopt], \xi_s)}^2 }}$
     		\smallskip
    		\STATE $\lead = \argmin_{\point \in \compact} \ip{ \curr[\wgrad] \sgrad (\curr[\stateopt], \xi_t) }{ \point } + \frac{\curr[\wgrad] \curr[\wavg] }{2 \curr[\wavgsum]} \ip{\shess ( \curr[\stateopt] , \xi_t) ( \point - \curr ) }{\point - \curr} + \frac{1}{2 \curr[\stepalt]} \norm{ \point - \curr }^2$
    		\STATE $\, \next= \argmin_{\point \in \compact} \ip{ \curr[\wgrad] \sgrad (\lead[\stateavg], \xi_{t+\frac{1}{2}}) }{ \point} + \frac{1}{2 \curr[\stepalt]} \norm{\point - \curr}^2$
    		\smallskip
    	\ENDFOR
    \end{algorithmic}
\end{algorithm}

\begin{repproposition}{prop:template-inequality}
    Let $\{ \lead \}_{\run = \start}^{\nRuns}$ be generated by \algref{alg:explicit}, run with a non-increasing step-size sequence $\curr[\stepalt]$ and non-decreasing sequences of weights $\curr[\wgrad], \curr[\wavg] \geq 1$ such that $\curr[\wgrad] / \curr[\wavg]$ is also non-increasing. Then, the following guarantee holds:
    \begin{align*}
        \mathbb E \Reg(\sol[\point]) \leq \frac{1}{2} \Expect{ \frac{3\eucdiam^2}{\gamma_{T+1}} + \sum_{\run=\start}^\nRuns \next[\stepalt] \curr[\wgrad]^2 \norm{\sgrad(\lead[\stateavg], \lead[\xi]) - \Fs(\lead[\stateavg]; \curr[\stateopt], \curr[\xi])}^2 - \frac{\norm{\lead - \curr}^2}{\next[\stepalt]} }
    \end{align*}
\end{repproposition}
\begin{proof}
     We take off from the optimality conditions associated with each update sequence for our explicit algorithm \method (\algref{alg:explicit}). Optimality condition for $\lead$ implies for any $z_0 \in \compact$,
    \begin{equation} \label{eq:optimality-1}
    \begin{aligned}
        &\ip{ \curr[\wgrad] \sgrad( \curr[\stateopt], \curr[\xi] ) + \curr[\wgrad] \frac{\curr[\wavg]}{\curr[\wavgsum]} \shess( \curr[\stateopt], \curr[\xi]) (\lead - \curr) }{ \lead - z_0 } \\
        &= 
        \ip{ \curr[\wgrad] \sgrad( \curr[\stateopt], \curr[\xi] ) + \curr[\wgrad] \shess( \curr[\stateopt], \curr[\xi]) (\lead[\stateavg] - \curr[\stateopt]) }{ \lead - z_0 } \\
        &= 
        \ip{ \curr[\wgrad] \Fs(\lead[\stateavg]; \curr[\stateopt], \curr[\xi]) }{ \lead - z_0 } \\
        &\leq 
        \frac{1}{\curr[\stepalt]} \ip{ \lead - \curr }{ z_0 - \lead }\\
        &= 
        \frac{1}{2\curr[\stepalt]} \br{ \norm{\curr - z_0}^2 - \norm{\lead - z_0}^2 - \norm{\lead - \curr}^2 }
    \end{aligned}
    \end{equation}
    Similarly, optimality of $\next$ update yields for any $z_1 \in \mathcal X$,
	\begin{equation} \label{eq:optimality-2}
    \begin{aligned}
		\ip{ \curr[\wgrad] \sgrad(\lead[\stateavg], \lead[\xi]) }{ \next - z_1 } &\leq \frac{1}{2 \curr[\stepalt]} \ip{ \next - \curr }{ z_1 - \next }\\
		&= \frac{1}{2\curr[\stepalt]} \br{ \norm{ \curr - z_1 }^2 - \norm{ \next - z_1 }^2 - \norm{ \next - \curr}^2 }
	\end{aligned}
    \end{equation}
    First, we will set $z_1 = \sol[\point]$ to establish the telescoping summation over $\norm{ \curr - \sol[\point] }^2 - \norm{ \next - \sol[\point] }^2$. Then, we will simply align the above expression with the regret as follows,
    \begin{equation} \label{eq:regret-intermediate}
    \begin{aligned}
        \langle \curr[\wgrad] \sgrad(\lead[\stateavg], \lead[\xi]), &\,\lead - x^\star \rangle\\
        &= \ip{ \curr[\wgrad] \sgrad(\lead[\stateavg], \lead[\xi]) }{ \lead - \next } + \ip{ \curr[\wgrad] \sgrad(\lead[\stateavg], \lead[\xi]) }{ \next - x^\star }\\
        &\leq \ip{ \curr[\wgrad] \sgrad(\lead[\stateavg], \lead[\xi]) }{ \lead - \next }\\
        &\quad+ \frac{1}{2\curr[\stepalt]} \br{ \norm{ \curr - x^\star }^2 - \norm{ \next - x^\star }^2 - \norm{ \next -  \curr}^2 }
    \end{aligned}
    \end{equation}
    Now, observe that setting $z_0 = \next$ in Eq.~\eqref{eq:optimality-1} and rearranging we have
    \begin{align*}
        &- \frac{1}{2\curr[\stepalt]} \norm{\next - \curr}^2\\ 
        &\leq - \ip{ \curr[\wgrad] \Fs(\lead[\stateavg]; \curr[\stateopt], \curr[\xi]) }{ \lead - \next } - \frac{1}{2\curr[\stepalt]} \br{ \norm{\lead - \next}^2 + \norm{\lead - \curr}^2 }
    \end{align*}
    Plugging the above expression into Eq.~\eqref{eq:regret-intermediate} and summing over $t=1, ..., T$, we will obtain,
	\begin{align*}
	    \sum_{\run=\start}^\nRuns &\langle \curr[\wgrad] \sgrad(\lead[\stateavg], \lead[\xi]), \lead - x^\star \rangle\\
	    &\leq \sum_{\run=\start}^\nRuns \curr[\wgrad] \ip{ \sgrad(\lead[\stateavg], \lead[\xi]) - \Fs(\lead[\stateavg]; \curr[\stateopt], \curr[\xi]) }{ \lead - \next }\\
	    &\quad+ \sum_{\run=\start}^\nRuns \frac{1}{2\curr[\stepalt]} \br{ \norm{ \curr - x^\star }^2 - \norm{ \next - x^\star }^2 - \norm{\lead - \next}^2 - \norm{\lead - \curr}^2 }
	\end{align*}
	First off, we bound the inner product term using Cauchy-Schwarz and a slight generalization of Young's inequality \citep{rakhlin2013optimization}
	\begin{align*}
			\sum_{t=1}^{T}&\, \curr[\wgrad] \ip{\sgrad(\lead[\stateavg], \lead[\xi]) - \Fs(\lead[\stateavg]; \curr[\stateopt], \curr[\xi])}{\lead - \next}\\
			&\leq \sum_{t=1}^{T} \curr[\wgrad] \norm{\sgrad(\lead[\stateavg], \lead[\xi]) - \Fs(\lead[\stateavg]; \curr[\stateopt], \curr[\xi])} \norm{\lead - \next} \\
			&\leq \frac{1}{2} \sum_{t=1}^{T} \next[\stepalt] \curr[\wgrad]^2 \norm{\sgrad(\lead[\stateavg], \lead[\xi]) - \Fs(\lead[\stateavg]; \curr[\stateopt], \curr[\xi])}^2 + \frac{1}{\next[\stepalt]} \norm{\lead - \next}^2.
	\end{align*}
	We merge the expressions together,
	\begin{align*}
	    \sum_{\run=\start}^\nRuns &\langle \curr[\wgrad] \sgrad(\lead[\stateavg], \lead[\xi]), \lead - x^\star \rangle\\
	    &\leq \frac{1}{2} \sum_{\run=\start}^\nRuns \next[\stepalt] \curr[\wgrad]^2 \norm{\sgrad(\lead[\stateavg], \lead[\xi]) - \Fs(\lead[\stateavg]; \curr[\stateopt], \curr[\xi])}^2 + \frac{1}{\next[\stepalt]} \norm{\lead - \next}^2\\
	    &\quad+ \sum_{\run=\start}^\nRuns \frac{1}{2\curr[\stepalt]} \br{ \norm{ \curr - x^\star }^2 - \norm{ \next - x^\star }^2 - \norm{\lead - \next}^2 - \norm{\lead - \curr}^2 }
	\end{align*}
    It is important that we invoke generalized Young's inequality with step-size at time $t+1$. Since the step-size lags one iteration behind, $\curr[\stepalt]$ does not include $\norm{\sgrad(\lead[\stateavg], \lead[\xi]) - \Fs(\lead[\stateavg]; \curr[\stateopt], \curr[\xi])}^2$ and this would pose some problems in the later stages of the proof. Hence, we add/subtract $\frac{1}{\next[\stepalt]} \norm{\lead - \curr}^2$ and regroup the terms,
    \begin{align*}
        \sum_{\run=\start}^\nRuns &\langle \curr[\wgrad] \sgrad(\lead[\stateavg], \lead[\xi]), \lead - x^\star \rangle\\
	    &\leq \frac{1}{2} \sum_{\run=\start}^\nRuns \next[\stepalt] \curr[\wgrad]^2 \norm{\sgrad(\lead[\stateavg], \lead[\xi]) - \Fs(\lead[\stateavg]; \curr[\stateopt], \curr[\xi])}^2 - \frac{1}{\next[\stepalt]} \norm{\lead - \curr}^2\\
	    &\quad+ \frac{1}{2} \sum_{\run=\start}^\nRuns \br{ \frac{1}{\next[\stepalt]} - \frac{1}{\curr[\stepalt]} } \br{ \norm{\lead - \next}^2 + \norm{\lead - \curr}^2 } \\
	    &\quad+ \frac{1}{2} \sum_{\run=\start}^\nRuns \frac{1}{\curr[\stepalt]} \br{ \norm{ \curr - x^\star }^2 - \norm{ \next - x^\star }^2 } \\
	    &\leq \frac{1}{2} \sum_{\run=\start}^\nRuns \next[\stepalt] \curr[\wgrad]^2 \norm{\sgrad(\lead[\stateavg], \lead[\xi]) - \Fs(\lead[\stateavg]; \curr[\stateopt], \curr[\xi])}^2 - \frac{1}{\next[\stepalt]} \norm{\lead - \curr}^2\\
	    &\quad+ \frac{\norm{\init - \sol[\point]}^2}{2 \gamma_1} +  \frac{1}{2} \sum_{t=1}^{T-1} \br{ \frac{1}{\next[\stepalt]} - \frac{1}{\curr[\stepalt]} } \norm{ \next - x^\star }^2 + \eucdiam^2 \sum_{\run=\start}^\nRuns \br{ \frac{1}{\next[\stepalt]} - \frac{1}{\curr[\stepalt]} } \\
	    &\leq \frac{3\eucdiam^2}{2 \gamma_{T+1}} + \frac{1}{2} \sum_{\run=\start}^\nRuns \next[\stepalt] \curr[\wgrad]^2 \norm{\sgrad(\lead[\stateavg], \lead[\xi]) - \Fs(\lead[\stateavg]; \curr[\stateopt], \curr[\xi])}^2 - \frac{1}{2} \sum_{\run=\start}^\nRuns \frac{1}{\next[\stepalt]} \norm{\lead - \curr}^2
    \end{align*}
    where we have rewritten the telescoping summation for $\norm{ \curr - x^\star }^2 - \norm{ \next - x^\star }^2$ and used that $\eucdiam^2 = \sup_{x,y \in \compact} \norm{x - y}^2 $ (diameter of the constraint set) to obtain the second inequality. The final line follows from telescoping the summations, plugging in the diameter $\eucdiam$ and rearranging the resulting terms.
    
    Now, what remains is to obtain the (expected) regret from $\sum_{\run=\start}^\nRuns \langle \curr[\wgrad] \sgrad(\lead[\stateavg], \lead[\xi]), \lead - x^\star \rangle$. Recall the definitions of $\curr[\mathcal F] = \sigma( 
\xi_1, \xi_{1 + \frac{1}{2}}, \cdots, \curr[\xi] )$ and $\lead[\mathcal F] = \sigma( \xi_1, \xi_{1 + \frac{1}{2}}, \cdots, \curr[\xi], \lead[\xi] )$ from Table~\ref{tbl:app_notation}. Taking expectation over all randomness,
	\begin{align*}
	    &\mathbb E \Bigg[\sum_{\run=\start}^\nRuns\langle \curr[\wgrad] \sgrad(\lead[\stateavg], \lead[\xi]), \lead - x^\star \rangle \Bigg]\\
	    &= 
	    \Expect{ \sum_{\run=\start}^\nRuns\curr[\wgrad] \langle \sgrad(\lead[\stateavg], \lead[\xi]) - \nabla \obj( \lead[\stateavg] ), \lead - x^\star \rangle + \curr[\wgrad] \langle \nabla \obj( \lead[\stateavg] ), \lead - x^\star \rangle } \\
	    &= 
	    \Expect{ \sum_{\run=\start}^\nRuns \Expect{ \curr[\wgrad] \langle \sgrad(\lead[\stateavg], \lead[\xi]) - \nabla \obj( \lead[\stateavg] ), \lead - x^\star \rangle \mid \curr[\mathcal F] } }\\
	    &\quad+ \Expect{ \sum_{\run = \start}^{\nRuns} \curr[\wgrad] \langle \nabla \obj( \lead[\stateavg] ), \lead - x^\star \rangle } \\
	    &= 
	    \Expect{ \sum_{\run=\start}^\nRuns \curr[\wgrad] \langle \Expect{ \sgrad(\lead[\stateavg], \lead[\xi]) \mid \curr[\mathcal F] } - \nabla \obj( \lead[\stateavg] ), \lead - x^\star \rangle }\\
	    &\quad+ \Expect{ \sum_{\run = \start}^{\nRuns} \curr[\wgrad] \langle \nabla \obj( \lead[\stateavg] ), \lead - x^\star \rangle } \\
	    &= 
	    \Expect{ \sum_{\run = \start}^{\nRuns} \curr[\wgrad] \langle \nabla \obj( \lead[\stateavg] ), \lead - x^\star \rangle }
	\end{align*}
	We used towering property of expectation (equivalently total law of expectation) to have the second inequality, and the last line from the unbiasedness assumption of gradient oracles in Eq.~\eqref{eq:stochastic-oracle} such that $\Expect{ \sgrad(\lead[\stateavg], \lead[\xi]) \mid \curr[\mathcal F] } = \nabla \obj (\lead[\stateavg])$. Hence, we obtain that
	\begin{align*}
	    \Expect{ \Reg(\sol) } &= \Expect{ \sum_{\run = \start}^{\nRuns} \curr[\wgrad] \langle \nabla \obj( \lead[\stateavg] ), \lead - x^\star \rangle } \\
	    &= \mathbb E \Bigg[\sum_{\run=\start}^\nRuns\langle \curr[\wgrad] \sgrad(\lead[\stateavg], \lead[\xi]), \lead - x^\star \rangle \Bigg],
	\end{align*}
	which concludes the target result,
	\begin{align*}
	    \Expect{ \Reg(\sol[\point]) } \leq \frac{1}{2} \Expect{ \frac{3\eucdiam^2}{\gamma_{T+1}} + \sum_{\run=\start}^\nRuns \next[\stepalt] \curr[\wgrad]^2 \norm{\sgrad(\lead[\stateavg], \lead[\xi]) - \Fs(\lead[\stateavg]; \curr[\stateopt], \curr[\xi])}^2 - \frac{\norm{\lead - \curr}^2}{\next[\stepalt]} }
	\end{align*}
\end{proof}

\section{Technical Lemma for the Main Proofs} \label{app:numerical}




Before proceeding with the proofs of our main results, we need to establish the following technical result, due to \cite{MS10} and \cite{levy2018online}, which has been commonly used in the analysis of adaptive methods. We make use of it for the proof of \thmref{thm:implicit} and \thmref{thm:explicit}.

\begin{lemma}[\citealp{MS10}, \citealp{levy2018online}]
\label{lem:technical}
For all non-negative numbers $\alpha_{1},\dotsc \alpha_{\run}$, the following inequality holds:
\begin{equation}
\sqrt{\sum_{\run=1}^{\nRuns}\alpha_{\run}}\leq \sum_{\run=1}^{\nRuns}\dfrac{\alpha_{\run}}{\sqrt{\sum_{i=1}^{\run}\alpha_{i}}}\leq 2\sqrt{\sum_{\run=1}^{\nRuns}\alpha_{\run}}
\end{equation}
\end{lemma}

\section{\method: The First Universal Second-order Accelerated Method (\thmref{thm:explicit})} \label{app:explicit}

\begin{reptheorem}{thm:explicit}
     Let $\{ \lead \}_{\run = \start}^{\nRuns}$ be a sequence generated by \algref{alg:explicit}, run with the adaptive step-size policy \eqref{eq:adaptive-stepsize-explicit} and $\curr[\wgrad]=\run^{2},\curr[\wavg]=\run^{p}$ for $p \geq 2$. Assume that $\obj$ satisfies \eqref{eq:Hess-smooth}, and that Assumptions~\eqref{eq:stochastic-oracle} hold
     . Then, the following universal guarantee holds:
    \begin{align*}
        \obj (\lastlead[\stateavg]) - \obj (\sol[\point]) \leq O \br{ \frac{ { \frac{\eucdiam^2 + \stepalt^2}{\gamma} } \vargrad }{\sqrt{\nRuns}} + \frac{ { \frac{\eucdiam^3 + \eucdiam \stepalt^2}{\gamma} } \varhess }{\nRuns^{3/2}} + \frac{ \max \bc{ L {\frac{ \eucdiam^4 + \eucdiam \stepalt^3}{\gamma} }, \sqrt{\stepscale} { \frac{\eucdiam^2 + \stepalt^2}{\stepalt} } } }{\nRuns^3} }
    \end{align*}
    When $\stepalt = \eucdiam$, we obtain the target rate $O \br{ \frac{D \sigma_g}{\sqrt{\nRuns}} + \frac{D^2 \sigma_H}{\nRuns^{3/2}} + \frac{ \max \bc{\smooth \eucdiam^3, \sqrt{\stepscale} \eucdiam} }{\nRuns^3} }$.
\end{reptheorem}
\begin{proof}
    We take Proposition~\ref{prop:template-inequality} as our departure point for the analysis. After proving an offline regret bound, we will use \thmref{thm:conversion} to obtain the optimality gap from the regret bound. Recall the template regret bound,
    \begin{align*}
	    \mathbb E \Reg(\sol[\point]) \leq \frac{1}{2} \Expect{ \frac{3\eucdiam^2}{\gamma_{T+1}} + \sum_{\run=\start}^\nRuns \next[\stepalt] \curr[\wgrad]^2 \norm{\sgrad(\lead[\stateavg], \lead[\xi]) - \Fs(\lead[\stateavg]; \curr[\stateopt], \curr[\xi])}^2 - \frac{\norm{\lead - \curr}^2}{\next[\stepalt]} }
	\end{align*}
    Now, we want to unify the first two terms through numerical inequalities. We will write the \textit{second term in terms of the first term}. Due to \lemref{lem:technical}, we can upper the bound second term as,
    \begin{align*}
        \frac{1}{2} \sum_{\run=\start}^\nRuns \next[\stepalt] \curr[\wgrad]^2 &\norm{\sgrad(\lead[\stateavg], \lead[\xi]) - \Fs(\lead[\stateavg]; \curr[\stateopt], \curr[\xi])}^2 \\
        &= \frac{\gamma}{2} \sum_{\run=\start}^\nRuns \frac{ \curr[\wgrad]^2 \norm{\sgrad(\lead[\stateavg], \lead[\xi]) - \Fs(\lead[\stateavg]; \curr[\stateopt], \curr[\xi])}^2}{ \sqrt{ \stepscale + \sum_{\runalt=\start}^\run \iter[\wgrad]^2 \norm{\sgrad(\iterlead[\stateavg], \iterlead[\xi]) - \Fs(\iterlead[\stateavg]; \iter[\stateopt], \iter[\xi])}^2 } }\\
        &\leq \gamma \sqrt{ \stepscale + \sum_{\run=\start}^\nRuns \curr[\wgrad]^2 \norm{\sgrad(\lead[\stateavg], \lead[\xi]) - \Fs(\lead[\stateavg]; \curr[\stateopt], \curr[\xi])}^2 } - \frac{\gamma}{2 \sqrt{\stepscale}}
    \end{align*}
    Plugging this back into the original expression gives us
    \begin{align*}
        &\Expect{ \Reg(\sol[\point]) }\\
        &\leq \br{ \frac{3 \eucdiam^2}{2 \gamma} + \gamma } \sqrt{ \stepscale + \sum_{\run=\start}^\nRuns \curr[\wgrad]^2 \norm{\sgrad(\lead[\stateavg], \iter[\xi]) - \Fs(\lead[\stateavg]; \curr[\stateopt], \curr[\xi])}^2 } - \frac{1}{2} \sum_{\run=\start}^\nRuns \frac{1}{\next[\stepalt]} \norm{\lead - \curr}^2
    \end{align*}
    Next up, we will handle the negative term in the above expression.
    As we have discussed in the main text, the key for faster rates beyond $O(1/\nRuns^2)$ is understanding how to manipulate the negative term in the above expression. 
    A crucial part of our analysis is understanding the implications of second-order smoothness and how to unlock its potential. 
    This next derivation will demonstrate how \eqref{eq:Hess-smooth} allows for a more aggressive gradient weighting and in turn faster convergence rate implied by our generalized conversion technique. Next, we will relate the negative term to the positive terms using smoothness and primal averaging, similar to the approaches in \citep{wang2018acceleration, kavis2019universal}.
    \begin{align*}
        -\frac{\norm{\lead - \curr}^2}{\next[\stepalt]}
        &=
        - \frac{\eucdiam^2}{\eucdiam^2 \next[\stepalt]} \norm{ \lead - \curr }^2\\
        &\leq
        - \frac{1}{\eucdiam^2 \next[\stepalt]} \norm{ \lead - \curr }^4\\
        &=
        - \frac{1}{\eucdiam^2 \next[\stepalt]} \frac{\curr[\wavgsum]^4}{\curr[\wavg]^4} \norm{ \frac{\curr[\wavg]}{\curr[\wavgsum]} \lead - \frac{\curr[\wavg]}{\curr[\wavgsum]}\curr }^4\\
        &=
        - \frac{1}{\eucdiam^2 \next[\stepalt]} \frac{\curr[\wavgsum]^4}{\curr[\wavg]^4} \norm{ \frac{\curr[\wavg] \lead + \sum_{\runalt=\start}^{\run-1} \iter[\wavg] \iterlead }{\curr[\wavgsum]}  - \frac{\curr[\wavg] \curr + \sum_{\runalt=\start}^{\run-1} \iter[\wavg] \iterlead }{\curr[\wavgsum]} }^4\\
        &=
        - \frac{1}{\eucdiam^2 \next[\stepalt]} c^4 t^4 \norm{ \lead[\stateavg] - \curr[\stateopt] }^4\\
        &\leq
        - \frac{4 c^4 \run^4}{\smooth^2 \eucdiam^2 \next[\stepalt]} \norm{ \nabla f(\lead[\stateavg]) - \F(\lead[\stateavg]; \curr[\stateopt]) }^2
    \end{align*}
    First, notice that for any sequence $\curr[\wavg] = O(t^p)$ with $p \geq 0$, we have $\curr[\wavgsum] = \sum_{\runalt = \start}^{\run} \iter[\wavg] = O(\run^{p+1})$, which implies $\frac{\curr[\wavgsum]}{\curr[\wavg]} \leq c t$, where $c > 0$ is an absolute constant depending on how $\curr[\wavg]$ is defined. Then, we use the definitions of average sequences $\lead[\stateavg]$ and $\curr[\stateopt]$ to go from $\norm{ \lead - \curr }^4$ to $\norm{ \lead[\stateavg] - \curr[\stateopt] }^4$ to obtain equalities 3-5, and apply smoothness to obtain the last line. On a related note, we want to highlight the importance of optimistic weighted averaging that is central for obtaining the above expression. Since the averaged pairs $\lead[\stateavg]$ and $\curr[\stateopt]$ differ by only the last element, we can seamlessly relate $\norm{ \lead - \curr }$ to $\norm{ \lead[\stateavg] - \curr[\stateopt] }$.
    
    Now, we are at a position to explain how we will go beyond $O(1/\nRuns^2)$ convergence rate, which fundamentally depends on the gradient weights $\curr[\wgrad]$ and jointly relies on our generalized online-to-batch conversion in \thmref{thm:conversion}. The negative term above is monotonically decreasing (increases in magnitude) which is essential to (partially) control the growth of remaining positive term. More specifically, one can notice that in order to align the summands of the positive and negative term, the algebra dictates that we need to select $\curr[\wgrad] = O(\run^2)$, which implies $\curr[\wavg] = \Omega(t^2)$. Notice that our averaging and weighting parameters grow at least $O(\run)$ faster than the existing accelerated schemes for first-order smoothness, which grants the improved $O(1/\nRuns^3)$ rate. On the contrary, first-order smoothness would only allow $\run^2$ factor in front of the norm, leading to the slower rate.

    Due to (margin-wise) space constraints, we will use a slightly more compact notation for certain expressions. Let us first define a shorthand notation for noise in gradient and Hessian evaluations, respectively.
    \begin{equation} \label{eq:noise-variables}
    \begin{aligned}
        \curr[\eps] &= [\sgrad(\lead[\stateavg], \lead[\xi]) - \sgrad(\curr[\stateopt], \curr[\xi])] - [\nabla \obj(\lead[\stateavg]) - \nabla \obj(\curr[\stateopt])]\\
        \curr[\delta] &= \shess(\curr[\stateopt], \curr[\xi]) - \nabla^2 \obj(\curr[\stateopt])
    \end{aligned}
    \end{equation}
    Then, we define following deterministic/stochastic placeholders:
    \begin{equation} \label{eq:grad-shorthand}
        \begin{aligned}
            \curr[\nabla] &= \nabla \obj(\lead[\stateavg]) - \F(\lead[\stateavg], \curr[\stateopt])\\
            \curr[\tnabla] &= \sgrad (\lead[\stateavg], \lead[\xi]) - \Fs(\lead[\stateavg], \curr[\stateopt], \curr[\xi]) = \curr[\nabla] + \curr[\eps] - \curr[\delta] (\lead[\stateavg] - \curr[\stateopt])
        \end{aligned}
    \end{equation}
    Setting $\curr[\wgrad] = \run^2$, combining all the terms and introducing the compact notation,
    \begin{align*}
        &\sum_{\run=\start}^\nRuns \langle \curr[\wgrad] \sgrad(\lead[\stateavg], \lead[\xi]), \lead - x^\star \rangle\\
        &\leq \br{ \frac{3 \eucdiam^2}{2 \gamma} + \gamma } \sqrt{ \stepscale + \sum_{\run=\start}^\nRuns \curr[\wgrad]^2 \norm{\curr[\tnabla] }^2 } - \sum_{\run=\start}^\nRuns \frac{2 c^4}{\smooth^2 \eucdiam^2 \next[\stepalt]} \curr[\wgrad]^2 \norm{ \curr[\nabla] }^2
    \end{align*}
    At this point, we need to understand how to relate $\norm{\curr[\nabla]}^2$ and $\norm{\curr[\tnabla]}^2$ while treating the step-size $\next[\stepalt]$ accordingly. The issue is that the step-size is agnostic to deterministic oracle information since we accumulate $\norm{\curr[\tnabla]}^2$. From the perspective of step-size, we need to find a relevant, if not matching, lower bound for $\norm{\curr[\nabla]}^2$ and $\norm{\curr[\tnabla]}^2$. Indeed, we follow the ideas presented in \citep{kavis2019universal}, and begin by (trivially) lower bounding both terms with the same expression,
	\begin{equation} \label{eq:gradient-lower-bound}
	\begin{aligned}
	    \norm{\tnabla_t}^2 &\geq \min \bc{ \norm{\tnabla_t}^2, \norm{\nabla_t}^2 }\\
	    \norm{\nabla_t}^2 &\geq \min \bc{ \norm{\tnabla_t}^2, \norm{\nabla_t}^2 }\\
	\end{aligned}
	\end{equation}
	Now, we will decompose $\norm{\curr[\tnabla]}^2$ into $\norm{\nabla_t}^2$ and the noise terms. Using the definitions in Eq.~\eqref{eq:noise-variables} and~\eqref{eq:grad-shorthand} and applying triangular inequality with quadratic expansion,
	\begin{align}\label{eq:quadratic-bound}
	    \norm{\curr[\tnabla]}^2 \leq 2 \norm{\curr[\nabla]}^2 + 4 \norm{ \curr[\delta] (\lead[\stateavg] - \curr[\stateopt]) }^2 + 4\norm{ \curr[\eps] }^2
	\end{align}
	We can also have the following trivial upper bound,
	\begin{equation} \label{eq:trivial-bound}
	\begin{aligned}
	    \norm{\tnabla_t}^2 &\leq 2 \norm{\tnabla_t}^2\\ 
	    &\leq 2 \norm{\tnabla_t}^2 + 4 \norm{ \curr[\delta] (\lead[\stateavg] - \curr[\stateopt]) }^2 + 4\norm{ \curr[\eps] }^2
	\end{aligned}
	\end{equation}
	Let us simplify the relationship between the bounds in Eq.~\eqref{eq:quadratic-bound} and Eq.~\eqref{eq:trivial-bound}; if $\norm{\nabla_t}^2 \leq \norm{\tnabla_t}^2$, then Eq.~\eqref{eq:quadratic-bound} is tighter, otherwise Eq.~\eqref{eq:trivial-bound} is tighter. Hence, we could select the minimum of $\norm{\nabla_t}^2$ and $\norm{\tnabla_t}$:
	\begin{align}\label{eq:minimum-bound}
	    \norm{\tnabla_t}^2 \leq 2 \min \bc{ \norm{\tnabla_t}^2, \norm{\nabla_t}^2 } + 4 \norm{ \curr[\delta] (\lead[\stateavg] - \curr[\stateopt]) }^2 + 4\norm{ \curr[\eps] }^2
	\end{align}
	Using this intuition, we can construct a variable $\curr[\stepaltalt]$ that always upper bounds the step-size.
	\begin{align}
	    \curr[\stepaltalt] = \frac{\gamma}{\sqrt{ \stepscale + \sum_{\runalt=\start}^{\run-1} \iter[\wgrad]^2 \min \bc{ \norm{\iter[\tnabla]}^2, \norm{\iter[\nabla]}^2 } } }
	\end{align}
	It is immediate that $\gamma_t \leq \curr[\stepaltalt]$. Essentially, we will replace the terms $\norm{\nabla_t}^2$ and $\norm{\tnabla_t}^2$ with $\min \bc{ \norm{\tnabla_t}^2, \norm{\nabla_t}^2 }$, $\norm{ \curr[\delta] (\lead[\stateavg] - \curr[\stateopt]) }^2$ and $\norm{\eps_t}^2$. 
	\begin{align*}
	    &\Expect{ \Reg(\sol[\point]) }\\
        &\leq \mathbb E \Bigg[ { \frac{3 \eucdiam^2 + 2 \gamma^2}{{2} \gamma} } \sqrt{ \stepscale + \sum_{\run=\start}^\nRuns \curr[\wgrad]^2 \norm{\curr[\tnabla] }^2 } - \sum_{\run=\start}^\nRuns \frac{2 c^4}{\smooth^2 \eucdiam^2 \next[\stepalt]} \curr[\wgrad]^2 \norm{ \curr[\nabla] }^2 \Bigg] \\
        &\leq \mathbb E \Bigg[ { \frac{3 \eucdiam^2 + 2 \gamma^2}{{2} \gamma} } \sqrt{ \stepscale + \sum_{\run=\start}^\nRuns 2 \curr[\wgrad]^2  \min \bc{ \norm{\tnabla_t}^2, \norm{\nabla_t}^2 } + 4 \curr[\wgrad]^2 \norm{ \curr[\delta] (\lead[\stateavg] - \curr[\stateopt]) }^2 + 4 \curr[\wgrad]^2 \norm{ \curr[\eps] }^2 }\\
        &\quad - \sum_{\run=\start}^\nRuns \frac{2 c^4}{\smooth^2 \eucdiam^2 \next[\stepaltalt]} \curr[\wgrad]^2 \min \bc{ \norm{\tnabla_t}^2, \norm{\nabla_t}^2 } \Bigg]\\
        &\leq \mathbb E \Bigg[ { \frac{3 \eucdiam^2 + 2 \gamma^2}{\sqrt{2} \gamma} } \sqrt{ \stepscale + \sum_{\run=\start}^\nRuns \curr[\wgrad]^2  \min \bc{ \norm{\tnabla_t}^2, \norm{\nabla_t}^2 } } - \sum_{\run=\start}^\nRuns \frac{2 c^4 \curr[\wgrad]^2 }{\smooth^2 \eucdiam^2 \next[\stepaltalt]} \min \bc{ \norm{\tnabla_t}^2, \norm{\nabla_t}^2 }\\
        &\quad+ 2 \br{ \frac{3 \eucdiam^2}{2 \gamma} + \gamma } \sqrt{ \sum_{\run=\start}^\nRuns \curr[\wgrad]^2 \norm{ \curr[\delta] (\lead[\stateavg] - \curr[\stateopt]) }^2 } + 2\br{ \frac{3 \eucdiam^2}{2 \gamma} + \gamma } \sqrt{ \sum_{\run=\start}^\nRuns \curr[\wgrad]^2 \norm{ \curr[\eps] }^2 } ~~\Bigg]\\
        &\leq \mathbb E \Bigg[ { \frac{3 \eucdiam^2 + 2 \gamma^2}{\sqrt{2} \gamma^2} } \br{ \stepalt \sqrt{\stepscale} + \sum_{\run=\start}^\nRuns \next[\stepaltalt] \curr[\wgrad]^2 \min \bc{ \norm{\tnabla_t}^2, \norm{\nabla_t}^2 } } - \sum_{\run=\start}^\nRuns \frac{2 c^4 \curr[\wgrad]^2}{\smooth^2 \eucdiam^2 \next[\stepaltalt]} \min \bc{ \norm{\tnabla_t}^2, \norm{\nabla_t}^2 }\\
        &\quad+ 2 \br{ \frac{3 \eucdiam^2}{2 \gamma} + \gamma } \sqrt{ \sum_{\run=\start}^\nRuns \curr[\wgrad]^2 \norm{ \curr[\delta] (\lead[\stateavg] - \curr[\stateopt]) }^2 } + 2\br{ \frac{3 \eucdiam^2}{2 \gamma} + \gamma } \sqrt{ \sum_{\run=\start}^\nRuns \curr[\wgrad]^2 \norm{ \curr[\eps] }^2 } ~~\Bigg] \\ 
        &\leq { \frac{3 \eucdiam^2 + 2 \gamma^2}{\sqrt{2} \gamma} \sqrt{\stepscale} } + \mathbb E \Bigg[ \sum_{\run=\start}^\nRuns \br{ { \frac{3 \eucdiam^2 + 2 \gamma^2}{\sqrt{2} \gamma^2} } - \frac{2 c^4}{\smooth^2 \eucdiam^2 \next[\stepaltalt]^2} } \next[\stepaltalt] \curr[\wgrad]^2 \min \bc{ \norm{\tnabla_t}^2, \norm{\nabla_t}^2 }\\ 
        &\quad+ 2 \br{ \frac{3 \eucdiam^2}{2 \gamma} + \gamma } \sqrt{ \sum_{\run=\start}^\nRuns \curr[\wgrad]^2 \norm{ \curr[\delta] (\lead[\stateavg] - \curr[\stateopt]) }^2 } + 2\br{ \frac{3 \eucdiam^2}{2 \gamma} + \gamma } \sqrt{ \sum_{\run=\start}^\nRuns \curr[\wgrad]^2 \norm{ \curr[\eps] }^2 } ~~\Bigg]\\
	\end{align*}
    Next, we will simplify the first summation and eventually show that it has a finite, constant upper bound. First off, notice that $\br{ { \frac{3 \eucdiam^2 + 2 \gamma^2}{\sqrt{2} \gamma^2} } - \frac{2 c^4}{\smooth^2 \eucdiam^2 \next[\stepaltalt]^2} }$ is a decreasing quantity and we are interested in the time point at which it changes signs. Let us define,
    \begin{align*}
        \nRuns_0 = \max \bc{ t \in \mathbb Z \mid \br{ { \frac{3 \eucdiam^2 + 2 \gamma^2}{\sqrt{2} \gamma^2} } - \frac{2 c^4}{\smooth^2 \eucdiam^2 \next[\stepaltalt]^2} } \geq 0 }.
    \end{align*}
    This immediately implies that for any $t \leq \nRuns_0$,
    \begin{equation} \label{eq:gamma_T0}
    \begin{aligned}
        \frac{1}{\next[\stepaltalt]} &\leq \frac{ \smooth \eucdiam \sqrt{ 3 \eucdiam^2 + 2 \gamma^2 } }{ 2^{3/4} \gamma c^2 }.
    \end{aligned}
    \end{equation}
    There is a critical cut-off point for the possible values of $T_0$ depending on the value of $\stepscale$. When the initial step-size is small enough, i.e., $\stepscale$ is too large, then $T_0 < 0$. This occurs when $\stepscale \geq \frac{\smooth^2 \eucdiam^2 (3 \eucdiam^2 + 2\stepalt^2)}{2^{3/2} \stepalt^2 c^4 }$, which implies,
    \begin{align*}
        \mathbb E \Bigg[ \sum_{\run=\start}^\nRuns \br{ { \frac{3 \eucdiam^2 + 2 \gamma^2}{\sqrt{2} \gamma^2} } - \frac{2 c^4}{\smooth^2 \eucdiam^2 \next[\stepaltalt]^2} } \next[\stepaltalt] \curr[\wgrad]^2 \min \bc{ \norm{\tnabla_t}^2, \norm{\nabla_t}^2 } \Bigg] &\leq 0
    \end{align*}
    We get the same bound when $T_0 = 0$. For any other value of $T_0$, i.e., $T_0 > 0$, observe that the way we define $\nRuns_0$ enables us to \textit{upper bound} the summation up to $\nRuns$, with the summation up to $\nRuns_0$. Hence,
    \begin{align*}
        \frac{3 \eucdiam^2 + 2 \gamma^2}{\sqrt{2} \gamma} &\sqrt{\stepscale}+ \mathbb E \Bigg[ \sum_{\run=\start}^\nRuns \br{ { \frac{3 \eucdiam^2 + 2 \gamma^2}{\sqrt{2} \gamma^2} } - \frac{2 c^4}{\smooth^2 \eucdiam^2 \next[\stepaltalt]^2} } \next[\stepaltalt] \curr[\wgrad]^2 \min \bc{ \norm{\tnabla_t}^2, \norm{\nabla_t}^2 } \Bigg] \\
        &\leq { \frac{3 \eucdiam^2 + 2 \gamma^2}{\sqrt{2} \gamma} \sqrt{\stepscale} } + \mathbb E \Bigg[ \sum_{\run=\start}^{\nRuns_0} \br{ \frac{3 \eucdiam^2 + 2 \gamma^2}{\sqrt{2} \gamma^2} - \frac{2 c^4}{\smooth^2 \eucdiam^2 \next[\stepaltalt]^2} } \next[\stepaltalt] \curr[\wgrad]^2 \min \bc{ \norm{\tnabla_t}^2, \norm{\nabla_t}^2 } \Bigg]\\
        &\leq \frac{3 \eucdiam^2 + 2 \gamma^2}{\sqrt{2} \gamma} \sqrt{\stepscale} + \frac{3 \eucdiam^2 + 2 \gamma^2}{\sqrt{2} \gamma} \sum_{\run=\start}^{\nRuns_0} \frac{ \curr[\wgrad]^2 \min \bc{ \norm{\tnabla_t}^2, \norm{\nabla_t}^2 } }{ \sqrt{ \stepscale + \sum_{\runalt=\start}^{\run-1} \iter[\wgrad]^2 \min \bc{ \norm{\iter[\tnabla]}^2, \norm{\iter[\nabla]}^2 } } }\\
        &\leq \frac{3 \sqrt{2} \eucdiam^2 + 2 \sqrt{2} \gamma^2}{ \gamma} \sqrt{ \stepscale + \sum_{\run=\start}^{\nRuns_0} \curr[\wgrad]^2 \min \bc{ \norm{\tnabla_t}^2, \norm{\nabla_t}^2 } }\\
        &= \br{3 \sqrt{2} \eucdiam^2 + 2 \sqrt{2} \gamma^2} \frac{1}{ \stepaltalt_{\nRuns_0 + 1} }\\
        &\leq \frac{ \smooth \eucdiam \br{ 3 \eucdiam^2 + 2 \gamma^2 }^{3/2} }{ 2^{1/4} \gamma c^2 }
    \end{align*}
    To make sure we incorporate the effect of the initial step-size, we combine the bounds to get
    \begin{align*}
        \frac{3 \eucdiam^2 + 2 \gamma^2}{\sqrt{2} \gamma} &\sqrt{\stepscale}+ \mathbb E \Bigg[ \sum_{\run=\start}^\nRuns \br{ { \frac{3 \eucdiam^2 + 2 \gamma^2}{\sqrt{2} \gamma^2} } - \frac{2 c^4}{\smooth^2 \eucdiam^2 \next[\stepaltalt]^2} } \next[\stepaltalt] \curr[\wgrad]^2 \min \bc{ \norm{\tnabla_t}^2, \norm{\nabla_t}^2 } \Bigg] \\
        &\leq \frac{3\eucdiam^2 + 2 \stepalt^2}{2^{1/4} \stepalt} \max \bc{ \frac{\sqrt{\stepscale}}{2^{1/4}}, \frac{\smooth \eucdiam \sqrt{ 3 \eucdiam^2 + 2 \stepalt^2 }}{c^2} }
    \end{align*}
    This gives us the constant part of the regret, which will lead to the $O(1 / \nRuns^3)$ part of the convergence rate. Now, what remains is to handle the ``stochasticity''. 
    We will bound the remaining stochastic terms with respect to the stochastic gradient and the stochastic Hessian. Plugging the expected regret in to the bound and combining all the expressions together,
    \begin{align*}
        &\Expect{ \Reg(\sol[\point]) }\\
        &\leq { \frac{3 \eucdiam^2 + 2 \gamma^2}{ \gamma} } \Expect{  \sqrt{ \sum_{\run=\start}^\nRuns \curr[\wgrad]^2 \norm{ \curr[\delta] (\lead[\stateavg] - \curr[\stateopt]) }^2 } + \sqrt{ \sum_{\run=\start}^\nRuns \curr[\wgrad]^2 \norm{ \curr[\eps] }^2 } } + \frac{3\eucdiam^2 + 2 \stepalt^2}{2^{1/4} \stepalt} \max \bc{ \frac{\sqrt{\stepscale}}{2^{1/4}}, \frac{\smooth \eucdiam \sqrt{ 3 \eucdiam^2 + 2 \stepalt^2 }}{c^2} }\\
        &\leq \frac{3\eucdiam^2 + 2 \stepalt^2}{2^{1/4} \stepalt} \max \bc{ \frac{\sqrt{\stepscale}}{2^{1/4}}, \frac{\smooth \eucdiam \sqrt{ 3 \eucdiam^2 + 2 \stepalt^2 }}{c^2} } + { \frac{3 \eucdiam^2 + 2 \gamma^2}{ \gamma} } \br{ \sqrt{ \sum_{\run=\start}^\nRuns \Expect{ \curr[\wgrad]^2 \norm{ \curr[\delta]}^2 \norm{(\lead[\stateavg] - \curr[\stateopt]) }^2 } } }\\
        &\quad+ { \frac{3 \eucdiam^2 + 2 \gamma^2}{ \gamma} } \sqrt{ \sum_{\run=\start}^\nRuns \Expect{ \curr[\wgrad]^2 [ \norm{ \sgrad(\lead[\stateavg], \lead[\xi]) - \nabla \obj(\lead[\stateavg]) }^2 + \norm{ \sgrad(\curr[\stateopt], \curr[\xi]) - \nabla \obj(\curr[\stateopt]) }^2 ] } } \\
        &= \frac{3\eucdiam^2 + 2 \stepalt^2}{2^{1/4} \stepalt} \max \bc{ \frac{\sqrt{\stepscale}}{2^{1/4}}, \frac{\smooth \eucdiam \sqrt{ 3 \eucdiam^2 + 2 \stepalt^2 }}{c^2} } + { \frac{3 \eucdiam^2 + 2 \gamma^2}{ \gamma} }  \sqrt{ \eucdiam^2 \sum_{\run=\start}^\nRuns \Expect{ \curr[\wgrad]^2 \frac{\curr[\wavg]^2}{\curr[\wavgsum]^2}\Expect{ \norm{ \curr[\delta] }^2 \mid \curr[\mathcal F] } } }  \\
        &\quad+ { \frac{3 \eucdiam^2 + 2 \gamma^2}{ \gamma} } \sqrt{ \sum_{\run=\start}^\nRuns \curr[\wgrad]^2 \Expect{ \Expect{ \norm{ \sgrad(\lead[\stateavg], \lead[\xi]) - \nabla \obj(\lead[\stateavg]) }^2 \mid \curr[\mathcal F] } + \Expect{ \norm{ \sgrad(\curr[\stateopt], \curr[\xi]) - \nabla \obj(\curr[\stateopt]) }^2 \mid \beforelead[\mathcal F] } } } \\
    \end{align*}
    \begin{align*}
        &\leq { \frac{3 \eucdiam^2 + 2 \gamma^2}{ \gamma} } \br{ \sqrt{ \eucdiam^2 \varhess^2 \sum_{\run=\start}^\nRuns \curr[\wgrad]^2 \frac{\curr[\wavg]^2}{\curr[\wavgsum]^2} } + \sqrt{ 4 \vargrad^2 \sum_{\run=\start}^\nRuns \curr[\wgrad]^2 } } + \frac{3\eucdiam^2 + 2 \stepalt^2}{2^{1/4} \stepalt} \max \bc{ \frac{\sqrt{\stepscale}}{2^{1/4}}, \frac{\smooth \eucdiam \sqrt{ 3 \eucdiam^2 + 2 \stepalt^2 }}{c^2} } \\
        &\leq { \frac{3 \eucdiam^2 + 2 \gamma^2}{ \gamma} } \br{ \sqrt{ \frac{\eucdiam^2 \varhess^2}{c^2} \sum_{\run=\start}^\nRuns \curr[\wgrad]  } + 2 \vargrad T^{5/2} } + \frac{3\eucdiam^2 + 2 \stepalt^2}{2^{1/4} \stepalt} \max \bc{ \frac{\sqrt{\stepscale}}{2^{1/4}}, \frac{\smooth \eucdiam \sqrt{ 3 \eucdiam^2 + 2 \stepalt^2 }}{c^2} }\\
        &\leq \frac{3\eucdiam^2 + 2 \stepalt^2}{2^{1/4} \stepalt} \max \bc{ \frac{\sqrt{\stepscale}}{2^{1/4}}, \frac{\smooth \eucdiam \sqrt{ 3 \eucdiam^2 + 2 \stepalt^2 }}{c^2} } + { \frac{3 \eucdiam^3 + 2 \eucdiam \gamma^2}{ c \gamma} } \varhess \nRuns^{3/2} + { \frac{6 \eucdiam^2 + 4 \gamma^2}{ \gamma} } \vargrad \nRuns^{5/2}
    \end{align*}
    Before concluding the convergence proof, we would like to have a quick detour on the value of $c$. The value of $c$ is roughly between $[1/p, 1]$, where $p$ is the exponent of the averaging weight, $\curr[\wavg] = t^p$. For instance, when we pick $\curr[\wavg] = t^2$, we have $\run^3 / 3 \leq \curr[\wavgsum] \leq \run^3$; and when $\curr[\wavg] = t^3$, $\run^4 /4 \leq \curr[\wavgsum] \leq \run^4$. Hence, we can avoid its effect in the final bound.
    Running the above expression through \thmref{thm:conversion} we obtain,
    \begin{align*}
        \obj (\lastlead[\stateavg]) - \obj (\sol[\point]) \leq O \br{ \frac{ { \frac{\eucdiam^2 + \stepalt^2}{\gamma} } \vargrad }{\sqrt{\nRuns}} + \frac{ { \frac{\eucdiam^3 + \eucdiam \stepalt^2}{\gamma} } \varhess }{\nRuns^{3/2}} + \frac{ \max \bc{ L {\frac{ \eucdiam^4 + \eucdiam \stepalt^3}{\gamma} }, \sqrt{\stepscale} { \frac{\eucdiam^2 + \stepalt^2}{\stepalt} } } }{\nRuns^3} }
    \end{align*}
\end{proof}

\section{Implicit Accelerated Second-order Algorithm (\thmref{thm:implicit}) } \label{app:implicit}

In this section, we will provide the analysis of the implicit algorithm \eqref{eq:implicit} under deterministic oracles. To do so, we will first start with a corollary result based on Proposition~\ref{prop:template-inequality} that essentially proves the same template inequality under deterministic oracle model. In fact, one could easily show that Proposition~\ref{prop:template-inequality} holds exactly up to replacing stochastic evaluations $\sgrad(\cdot)$ and $\Fs(\cdot;\cdot)$ with $\nabla f(\cdot)$ and $\F(\cdot;\cdot)$. For completeness, we will formalize the aforementioned result in Proposition~\ref{prop:template-inequality-implicit} which follows the same steps as the proof of Proposition~\ref{prop:template-inequality}.

\begin{proposition} \label{prop:template-inequality-implicit}
    Let $\{ \lead \}_{\run = \start}^{\nRuns}$ be generated by \eqref{eq:implicit}, run with a non-increasing step-size sequence $\curr[\stepalt]$ and non-decreasing sequences of weights $\curr[\wgrad], \curr[\wavg] \geq 1$ such that $\curr[\wgrad] / \curr[\wavg]$ is also non-increasing. Then, the following guarantee holds:
    \begin{align*}
        \Reg(\sol[\point]) \leq \frac{1}{2} \br{\frac{3\eucdiam^2}{\gamma_{T+1}} + \sum_{\run=\start}^\nRuns \next[\stepalt] \curr[\wgrad]^2 \norm{\nabla \obj(\lead[\stateavg]) - \F(\lead[\stateavg]; \curr[\stateopt])}^2 - \frac{\norm{\lead - \curr}^2}{\next[\stepalt]} }.
    \end{align*}
\end{proposition}
\begin{proof}
    The proof of this theorem is analogous to that of Proposition~\ref{prop:template-inequality} in Section~\ref{app:template-inequality}, up to replacing the stochastic feedback with the deterministic oracle calls.
\end{proof}

\begin{reptheorem}{thm:implicit}
    Let $\{ \lead \}_{\run = \start}^{\nRuns}$ be a sequence generated by \eqref{eq:implicit}, run with the adaptive step-size policy \eqref{eq:adaptive-stepsize-implicit} where $\curr[\wgrad]=\run^{2}$, $\curr[\wavg]=\run^{3}$. Assume that $\obj$ satisfies \eqref{eq:Hess-smooth} and denote the diameter of the set as $\eucdiam$. Then, the following guarantee holds:
     \begin{align*}
        \obj (\lastlead[\stateavg]) - \obj(\sol) \leq O \br{ \frac{ \max \bc{ \sqrt{\stepscale} \frac{\eucdiam^2}{\stepalt}, L \frac{\eucdiam^4 + \eucdiam \stepalt^3}{\stepalt} } }{\nRuns^3} }
    \end{align*}
    When $\stepalt = D$, we obtain the converge rate $O \br{ \frac{ \max \bc{\smooth \eucdiam^3, \sqrt{\stepscale} D }}{\nRuns^3}}$.
\end{reptheorem}
\begin{proof}
    We will initiate our proof at template regret inequality as we proved in Proposition~\ref{prop:template-inequality-implicit}. Our overall strategy is straightforward; we first prove a constant upper bound for the offline weighted regret, then make use of the conversion result in \thmref{thm:conversion} to obtain a convergence rate of order $O(1/ \nRuns^3)$. 
    
    Due to Proposition~\ref{prop:template-inequality-implicit} we have,
    \begin{align*}
        \Reg(\sol[\point]) \leq \frac{1}{2} \br{\frac{3\eucdiam^2}{\gamma_{T+1}} + \sum_{\run=\start}^\nRuns \next[\stepalt] \curr[\wgrad]^2 \norm{\nabla \obj(\lead[\stateavg]) - \F(\lead[\stateavg]; \curr[\stateopt])}^2 - \frac{\norm{\lead - \curr}^2}{\next[\stepalt]} }
    \end{align*}
    We will merge the first two terms and express the first term in the form of the second one using \lemref{lem:technical}. Observe that for the proof of Theorem~\ref{thm:explicit}, we did the opposite and converted the summation into the form of the first term, $\frac{3 \eucdiam^2}{2 \stepalt}$.
    \begin{align*}
        &\Reg(\sol[\point]) \\[1mm]
        &\leq \frac{3\eucdiam^2}{2 \stepalt} \sqrt{ \stepscale + \sum_{\run=\start}^\nRuns \curr[\wgrad]^2 \norm{\nabla \obj(\lead[\stateavg]) - \F(\lead[\stateavg]; \curr[\stateopt])}^2 } \\
        &\quad+ \frac{1}{2} \sum_{\run=\start}^\nRuns \next[\stepalt] \curr[\wgrad]^2 \norm{\nabla \obj(\lead[\stateavg]) - \F(\lead[\stateavg]; \curr[\stateopt])}^2 - \frac{\norm{\lead - \curr}^2}{\next[\stepalt]} \\
        &\leq \frac{3\eucdiam^2 \sqrt{\stepscale}}{2 \stepalt} + \frac{3\eucdiam^2}{2 \stepalt} \sum_{\run=\start}^\nRuns \frac{ \curr[\wgrad]^2 \norm{\nabla \obj(\lead[\stateavg]) - \F(\lead[\stateavg]; \curr[\stateopt])}^2 }{ \sqrt{ \stepscale + \sum_{\runalt=\start}^\run \iter[\wgrad]^2 \norm{\nabla \obj(\iterlead[\stateavg]) - \F(\iterlead[\stateavg]; \iter[\stateopt])}^2 } } \\
        &\quad+ \frac{1}{2} \sum_{\run=\start}^\nRuns \next[\stepalt] \curr[\wgrad]^2 \norm{\nabla \obj(\lead[\stateavg]) - \F(\lead[\stateavg]; \curr[\stateopt])}^2 - \frac{\norm{\lead - \curr}^2}{\next[\stepalt]} \\
        &= \frac{3\eucdiam^2 \sqrt{\stepscale} }{2 \stepalt} + \frac{3\eucdiam^2}{2 \stepalt^2} \sum_{\run=\start}^\nRuns \next[\stepalt] \curr[\wgrad]^2 \norm{\nabla \obj(\lead[\stateavg]) - \F(\lead[\stateavg]; \curr[\stateopt])}^2 \\
        &\quad+ \frac{1}{2} \sum_{\run=\start}^\nRuns \next[\stepalt] \curr[\wgrad]^2 \norm{\nabla \obj(\lead[\stateavg]) - \F(\lead[\stateavg]; \curr[\stateopt])}^2 - \frac{\norm{\lead - \curr}^2}{\next[\stepalt]} \\
        &= \frac{3\eucdiam^2 \sqrt{\stepscale} }{2 \stepalt} + \frac{1}{2} \sum_{\run=\start}^\nRuns { \frac{3\eucdiam^2 + \stepalt^2}{\stepalt^2}} \next[\stepalt] \curr[\wgrad]^2 \norm{\nabla \obj(\lead[\stateavg]) - \F(\lead[\stateavg]; \curr[\stateopt])}^2 - \frac{\norm{\lead - \curr}^2}{\next[\stepalt]},
    \end{align*}
    where we obtain the second inequality due to \lemref{lem:technical} and the last two lines follow from the definition of the step-size in Eq.~\eqref{eq:adaptive-stepsize-implicit} and appropriate regrouping.
    Similar to the proof in the explicit algorithm, we upper bound the negative term using appropriate averaging constants and smoothness. 
    \begin{align*}
        -\frac{\norm{\lead - \curr}^2}{\next[\stepalt]}
        &\leq
        - \frac{4 c^4}{\smooth^2 \eucdiam^2 \next[\stepalt]} \run^4 \norm{ \nabla f(\lead[\stateavg]) - \F(\lead[\stateavg]; \curr[\stateopt]) }^2\\
    \end{align*}
    Setting $\curr[\wgrad] = \run^2$, plugging the bound on the negative term into the original expression we have,
    \begin{align} \label{eq:intermediate-implicit}
         &\leq \frac{3\eucdiam^2 \sqrt{\stepscale} }{2 \stepalt} + \frac{1}{2} \sum_{\run=\start}^\nRuns \br{ \frac{3\eucdiam^2 + \stepalt^2}{\stepalt^2} - \frac{4 c^4}{\smooth^2 \eucdiam^2 \next[\stepalt]^2}} \next[\stepalt] \curr[\wgrad]^2 \norm{\nabla \obj(\lead[\stateavg]) - \F(\lead[\stateavg]; \curr[\stateopt])}^2
    \end{align}
    Our main objective is to show that the above summation is summable so we could show the constant upper bound for the offline regret, hence the acceleration. First off, notice that $\br{ \frac{3\eucdiam^2 + \stepalt^2}{\stepalt^2} - \frac{4 c^4}{\smooth^2 \eucdiam^2 \next[\stepalt]^2}}$ is a non-increasing quantity and we are interested in the time point at which this quantity becomes negative. For that reason, we define the following time point,
    \begin{align*}
        \nRuns_0 = \max \bc{ t \in \mathbb Z \mid \br{ \frac{3\eucdiam^2 + \stepalt^2}{\stepalt^2} - \frac{4 c^4}{\smooth^2 \eucdiam^2 \next[\stepalt]^2}} \geq 0 }.
    \end{align*}
    This immediately implies that for any $\run \leq \nRuns_0$,
    \begin{equation} \label{eq:gamma_T0_implicit}
    \begin{aligned}
        \frac{1}{\next[\stepalt]} &\leq \frac{ \smooth \eucdiam \sqrt{ 3 \eucdiam^2 + \gamma^2 } }{ 2 \gamma c^2 }.
    \end{aligned}
    \end{equation}
    To paint a complete picture, we would like to have a brief discussion on the possible values for $T_0$.
    \begin{enumerate}
        \item $T_0 \leq 0$ implies that the step-size is small enough from the very beginning and that the summation term in Eq.~\eqref{eq:intermediate-implicit} is always bounded by a constant, which immediately implies constant regret and $O(1/\nRuns^3)$ rate.
        \item $T_0 = \infty$ implies that the step-size is always lower bounded by the \textit{inverse of the constant on the right-hand side} of Eq.\eqref{eq:gamma_T0_implicit}. This is equivalent to saying $\sum_{\run=\start}^\infty \curr[\wgrad]^2 \norm{\nabla \obj(\lead[\stateavg]) - \F(\lead[\stateavg]; \curr[\stateopt])}^2 \leq C$ for some constant $C$, which in turn ensures that the summation in Eq.~\eqref{eq:intermediate-implicit} is summable. Once again, we will have the constant regret and $O(1/\nRuns^3)$ rate.
        \item When $T_0$ is a finite positive integer, we can upper bound the summation in Eq.~\eqref{eq:intermediate-implicit} with the same summation up to iteration $T_0$. Note that it is not important whether $\nRuns$ is larger or smaller than $T_0$, as the summands change sign and become negative after $T_0$.
    \end{enumerate}
    Same as in the proof of \method, we need to understand the effect of the initial step-size choice due to $\stepscale$. Imagine the case $\sqrt{\stepscale} \geq \frac{ \smooth \eucdiam \sqrt{ 3 \eucdiam^2 + \gamma^2 } }{ 2 \gamma c^2 }$. This implies that $T_0 < 0$ and that the step-size is already small enough to make the summation negative from the first step onwards. In that scenario, the condition in Eq.~\eqref{eq:gamma_T0_implicit} doesn't hold so we should consider the effect of this initial setup for the final bound.
    For the case when $T_0 > 0$, we can safely unify all the 3 cases above and simply upper bound the expression in Eq.~\eqref{eq:intermediate-implicit} by rewriting the summation up to $T_0$. Therefore,
    \begin{align*}
        &\leq \frac{3\eucdiam^2 \sqrt{\stepscale}}{2 \stepalt} + \frac{1}{2} \sum_{\run=\start}^\nRuns \br{ \frac{3\eucdiam^2 + \stepalt^2}{\stepalt^2} - \frac{4 c^4}{\smooth^2 \eucdiam^2 \next[\stepalt]^2}} \next[\stepalt] \curr[\wgrad]^2 \norm{\nabla \obj(\lead[\stateavg]) - \F(\lead[\stateavg]; \curr[\stateopt])}^2 \\
        &\leq \frac{3\eucdiam^2 + \stepalt^2}{2 \stepalt} \sqrt{\stepscale} + \frac{3\eucdiam^2 + \stepalt^2}{2 \stepalt^2} \sum_{\run=\start}^{\nRuns_0} \next[\stepalt] \curr[\wgrad]^2 \norm{\nabla \obj(\lead[\stateavg]) - \F(\lead[\stateavg]; \curr[\stateopt])}^2 \\
        &= \frac{3\eucdiam^2  + \stepalt^2}{2 \stepalt} \sqrt{\stepscale} + \frac{3\eucdiam^2 + \stepalt^2}{2 \stepalt} \sum_{\run=\start}^{\nRuns_0} \frac{ \curr[\wgrad]^2 \norm{\nabla \obj(\lead[\stateavg]) - \F(\lead[\stateavg]; \curr[\stateopt])}^2 }{ \sqrt{ \stepscale + \sum_{\runalt=\start}^\run \iter[\wgrad]^2 \norm{\nabla \obj(\iterlead[\stateavg]) - \F(\iterlead[\stateavg]; \iter[\stateopt])}^2 } }  \\
        &\leq \frac{3\eucdiam^2 + \stepalt^2}{\stepalt} \sqrt{ \stepscale + \sum_{\run=\start}^{\nRuns_0} \curr[\wgrad]^2 \norm{\nabla \obj(\lead[\stateavg]) - \F(\lead[\stateavg]; \curr[\stateopt])}^2 } \\
        &= \br{3\eucdiam^2 + \stepalt^2} \frac{1}{\stepalt_{T_0 + 1}} \\
        &\leq \frac{ \smooth \eucdiam \br{ 3 \eucdiam^2 + \gamma^2 }^{3/2} }{ 2 \gamma c^2 }
    \end{align*}
    We combine the case for $T_0 < 0$ with the one above to established the constant regret bound
    \begin{align*}
        \Reg(\sol) \leq O \br{ \max \bc{ \sqrt{\stepscale} \frac{\eucdiam^2}{\stepalt}, L \frac{\eucdiam^4 + \eucdiam \stepalt^3}{\stepalt} } }
    \end{align*}
    Plugging this result in its place we obtain the convergence rate,
    \begin{align*}
        \obj (\lastlead[\stateavg]) - \obj(\sol) \leq O \br{ \frac{ \max \bc{ \sqrt{\stepscale} \frac{\eucdiam^2}{\stepalt}, L \frac{\eucdiam^4 + \eucdiam \stepalt^3}{\stepalt} } }{\nRuns^3} }
    \end{align*}
\end{proof}

\end{document}